\documentclass[a4paper,11pt,oneside,reqno]{amsart}
\usepackage[T1]{fontenc}
\usepackage{color}
\usepackage{amsfonts}
\usepackage{amssymb, bbm, dsfont}
\usepackage{amsmath,amsthm,amsfonts,amssymb}
\usepackage[left=2.5cm,right=2.5cm, top=3.5cm, bottom=3.5cm]{geometry}
\usepackage[dvips]{graphicx}
\usepackage[english]{babel}
\usepackage{enumerate}
\usepackage{mathtools}

\let\oldmb\mathbold
\protected\def\mathbold{\oldmb}

\numberwithin{equation}{section}

\newcommand{\me}{\mathbb{E}}
\newcommand\del[1]{}
\newcommand\dela[1]{}

\newcommand{\rA}{\mathrm{A}}

\newcommand\toup{\nearrow}

\renewcommand{\v}{\bv}

\newcommand{\bw}{\mathbf{w}}

\newcommand{\MO}{\mathcal{O}}

\newcommand{\bv}{\mathbf{v}}
\newcommand{\bu}{\mathbf{u}}

\newcommand{\el}{\mathbb{L}}
\newcommand{\ve}{\mathbf{V}}
\newcommand{\h}{\mathbf{H}}
\newcommand{\bh}{\mathbb{H}}
\newcommand{\bve}{\mathbf{V}}
\newcommand{\be}{\mathbf{E}}

\newcommand{\eps}{\varepsilon}
\newcommand{\rve}{\rVert}
\newcommand{\lve}{\lVert}

\newcommand{\EE}{\mathbb{E}}

\newcommand{\Lve}{\pmb{\Vert}}

\newcommand{\FF}{\mathbb{F}}

\newcommand{\Rb}[1]{{\mathbb{R}_{#1}}}
\newcommand\cadlag{c{\`a}dl{\`a}g }
\newcommand{\wie}{\widetilde{\eta}}

\newcommand{\RR}{\mathbb{R}}
\newcommand{\CO}{\mathcal{O}}
\renewcommand{\th}{\theta}
\newcommand{\bbe}{\mathbf{E}}
\newcommand{\dn}{{ \nabla}}
\newcommand{\di}{{\rm div\,}}
\renewcommand{\k}{\kappa}
\newcommand{\D}{\Delta}

\newcommand{\CZ}{\mathcal{Z}}
\newcommand{\levy}{L\'evy }
\newcommand{\noe}{\Lve N\Lve_{\mathcal{L}(\be,\h)}}
\newcommand{\nov}{\Lve N\Lve_{\mathcal{L}(\ve,\ve^\ast)}}

\DeclareMathOperator{\Id}{Id}
\theoremstyle{plain}
\newtheorem{condition}{Condition}[section]
\newtheorem{assum}[condition]{Assumption}

\newtheorem{lem}{Lemma}[section]
\newtheorem{thm}[lem]{Theorem}
\newtheorem{prop}[lem]{Proposition}

\theoremstyle{definition}
\newtheorem{Def}[lem]{Definition}
\newtheorem{Rem}[lem]{Remark}
\title[Strong solution to stochastic hydrodynamical systems]{Strong solutions to stochastic hydrodynamical systems with multiplicative noise of jump type}

\author[H. Bessaih]{Hakima Bessaih}
\address[H. Bessaih]{Department of Mathematics, University of Wyoming, 1000 East University Avenue, Laramie WY 82071, United States,}
\author[E. Hausenblas]{Erika Hausenblas}
\author[P. Razafimandimby]{Paul Razafimandimby}
\address[E.~Hausenblas and P.~Razafimandimby]{Department of Mathematics and Information Technology, Montanuniversit\"at Leoben,
Fr. Josefstr. 18, 8700 Leoben, Austria}
\email[H.~Bessaih]{bessaih@uwyo.edu}
\email[E.~Hausenblas]{erika.hausenblas@unileoben.ac.at}
\email[P.~Razafimandimby]{paul.razafimandimby@unileoben.ac.at}
\begin{document}
\begin{abstract}
In this paper we prove the existence and uniqueness of maximal strong (in PDE sense) solution to several stochastic hydrodynamical systems on unbounded and bounded domains
of $\RR^n$, $n=2,3$. This maximal solution turns out to be a global one in the case of 2D stochastic hydrodynamical systems. Our framework is general in the sense that
it allows us to solve the Navier-Stokes equations, MHD equations, Magnetic B\'enard problems, Boussinesq model of the B\'enard
convection, Shell models of turbulence and the Leray-$\alpha$ model with jump type perturbation.
 Our goal is achieved by proving general results about the existence of maximal and global solution to an
abstract stochastic partial differential equations with locally
Lipschitz continuous coefficients. The method of the proofs are based on some truncation and fixed point methods.
\end{abstract}
\maketitle
\section{Introduction}
Stochastic Partial Differential Equations (SPDEs) are a powerful
tool for  understanding and
investigating mathematically hydrodynamic and turbulence theory. To
model turbulent fluids, mathematicians often use stochastic
equations obtained from adding a noise term in the dynamical
equations of the fluids. This approach is basically motivated by
Reynolds' work which stipulates that turbulent flows  are
composed of slow (deterministic) and fast (stochastic) components.
Recently by following the statistical approach of
turbulence theory,  Flandoli et al \cite{Flandoli2}, Kupiainen
\cite{KAP} confirm the importance of studying the stochastic
version of  fluids dynamics. Indeed, the
authors of  \cite{Flandoli2} pointed out that some rigorous
information on questions of turbulence theory might be obtained
from these stochastic versions.  It
is worth emphasizing that the presence of the stochastic term
(noise) in these models often leads to qualitatively new types of
behavior for the processes. Since the pioneering work of Bensoussan and Temam \cite{Bensoussan-Temam}, there has been an extensive literature on
stochastic Navier-Stokes equations with Wiener noise and related equations,
we refer to  \cite{BBBF06}, \cite{bensoussan},  \cite{bessaih}, \cite{Flandoli3}, \cite{Chueshov}, \cite{DEUGOUE2}, \cite{Flandoli+Gatarek}, \cite{ROZOVSKII1} amongst other.

In the last five years, there has been an extensive effort to
tackle SPDEs with Levy noise. There are several examples where the
Gaussian noise is not well suited to represent realistically
external forces. For example, if the ratio between the time scale
of the deterministic part and that of the stochastic noise is
large, then the temporal structure of the forcing in the course of
each event has no influence on the overall dynamics, and - at the
time scale of the deterministic process - the external forcing can
be modelled as a sequence of episodic instantaneous impulses. This
happens for example in Climatology (see, for instance,
\cite{MR2205118}). Often the noise observed by time series is
typically asymmetric, heavy-tailed and has  non trivial kurtosis.
These are all features which cannot be captured by a Gaussian
noise, but rather by a  \levy noise with appropriate parameters.
\levy randomness requires different techniques from the ones used
for Brownian motion and are less amenable to mathematical
analysis.
% Motivated by these reasons and the desire of developing mathematics many
%prominent mathematicians have started to spend a lot of effort on analyzing SPDES driven by \levy %noise.
%Particular example of results obtained so far are
We refer to \cite{ZB+EH+JZ},  \cite{Motyl-2},\cite{DongNS},  \cite{paulandme2} and \cite{Motyl} that deal with  stochastic hydrodynamical systems
driven by \levy type noise. Most of these articles are about the existence of solution which are weak in the PDEs sense.

In this paper, we are interested in proving the existence and uniqueness of maximal and global strong solution of \levy driven hydrodynamical systems such as
the Navier-Stokes equations (NS), Magnetohydrodynamics equations (MHD), Magnetic B\'enard problem (MB), Boussinesq model for B\'enard convection (BBC), Shell models of turbulence, and 3-D
Leray-$\alpha$ for Navier-Stokes equations. Here,  strong solutions should be understood in both the Probability and PDEs senses.
Our objectives are achieved by adopting the unified approach initiated and developed in \cite{Chueshov} and used later in
\cite{ZB+EH+JZ}. This approach is based on rewriting the various equations above into an abstract stochastic evolution equations in a Hilbert space
$\ve$ of the following form
\begin{equation}\label{Full-EQ-A}
 \bu(t)= \bu_0-\int_0^{t} \Big[A\bu(s)+F(\bu(s))\Big]ds+\int_0^{t}\int_Z  G(z,\bu(s))\wie(dz,ds),
\end{equation}
where $\int_Z  G(z,\bu(s))\wie(dz,ds)$ represents a global
Lipschitz continuous multiplicative noise of jump type. In Theorem
\ref{MAX-FULL-EQ} we give sufficient conditions (on $A$ and $F$)
for the existence and uniqueness of a maximal solution to
\eqref{Full-EQ-A}. Sufficient conditions for non-explosion of the
maximal solution in finite time is given in Theorem
\ref{GLO-FULL-EQ}.  These two theorems are our main results and
their assumptions are carefully chosen so that they are verified
by the NS, MHD, MB, BBC, Shell models and the Leray-$\alpha$
models. In Section \ref{EXAM} we borrow the examples and the
notations in \cite{Chueshov} and give a detailed account of the
applicability of our framework to the fluid models we cited in the
previous sentence.

The book \cite{Pes+Zab} contains several results about existence of solution to abstract SPDEs driven by \levy noise in Hilbert space setting, but the
hypotheses in this book do not cover the various hydrodynamical systems that we enumerated above. We also note that
while there are several results about the existence of
solution which are strong in PDEs sense for stochastic hydrodynamical systems perturbed by Wiener noise
(see, for instance, \cite{Bensoussan}, \cite{SLC-JAP} ,\cite{G-H+Z}, \cite{KIM}, \cite{MIKU}, \cite{ROZOVSKII1} and references therein), it seems that this is
the first paper treating the existence of strong (in PDE sense) solution for stochastic hydrodynamical systems with \levy noise. To prove our results we
closely follow \cite{Brz+Millet_2012} (see also \cite{SLC-JAP}) which in turn followed methods elaborated in two
papers by De Bouard and Debussche \cite{deBouard+Deb_1999,deBouard+Deb_2003}.

The layout of the present paper is as follows. In Section \ref{ABSTRACT}, we introduce the abstract stochastic evolution equation that our result will be based on.
In the very section we give the notations and standing assumptions, and prove some preliminary results that we are using throughout. Section \ref{MAX-GLO}
is devoted to the statements
and the proofs of our main results. We will mainly show that under the assumptions introduced in Section  \ref{ABSTRACT} the Eq. \eqref{Full-EQ-A}
admits a unique maximal local solution, and with additional conditions on $F$ and $G$ we prove that this maximal local solution turns out to be a global one.
 The results are obtained by use of cut-off and fixed point methods introduced in \cite{Brz+Millet_2012}. In Section \ref{EXAM} we give
 a detailed discussion
 on how our abstract results are used to solve the stochastic NS, MHD, MB, BBC, Shell models and Leray-$\alpha$ models
  driven by multiplicative noise of jump type. Most of the
  examples and notations in Section \ref{EXAM} are taken from
  \cite{Chueshov}.
  In appendix we prove the well-posedness of a linear stochastic evolution equations driven by compensated Poisson random measure which is very important for our analysis.

\section{Description of an abstract stochastic evolution equation}\label{ABSTRACT}
In this paper we give the necessary notations and standing assumptions used throughout the paper. We also prove some preliminary results that are
very important for our analysis.
\subsection{Notations and Preliminary results}\label{abstract-framework}

In this section we start with some notations,  then introduce the assumptions used throughout the paper and our abstract stochastic equation.

Let $(\ve, \lve \cdot \rve)$, $(\h, \lvert \cdot \rvert)$ and $(\be, \lve \cdot \rve_\ast)$ be three
separable and reflexive Hilbert spaces. The scaler product in $\h $ is denoted by $\langle u, v\rangle$ for any $u, v \in \h$.
The same symbol $\langle \phi , v \rangle$ will also be used to denoted duality pairing of $\phi\in \ve^\ast$ and
 $v \in \ve$.
We denote by $\mathcal{L}(Y_1,Y_2)$ be the space of bounded linear maps from a Banach space $Y_1$ into another Banach space $Y_2$.

For $T_2>T_1 \ge 0$ we set
\begin{equation}\label{eqn-X_T}
 X_{T_1,T_2}=L^\infty(T_1,T_2;\ve) \cap L^2(T_1,T_2;\be),
\end{equation}
with the norm $\lve u \rve_{X_{T_1, T_2}}$ defined by
\begin{equation}\label{eqn-X_T-norm}
 \Vert u\Vert_{X_{T_1,T_2}}^2= \sup_{s \in [T_1,T_2]} \lve \bu_n(s)\rve^2+\int_{T_1}^{T_2} \lve \bu_n(s) \rve_\ast^2\, ds.
 \end{equation}
For $T_1=0$ and $T_2=T>0$ we simply write $X_T:= X_{0, T}$.

Let $Y$ be a separable and complete metric space and $T>0$. The
space $\mathbf{D}([0,T];Y)$ denotes the space of all right continuous
functions $x:[0,1]\to Y$ with left limits. We equip $\mathbf{D}([0,T];Y)$ with the
Skorohod topology in which $\mathbf{D}([0,T];Y)$ is both separable and
complete. For more information about Skorohod space and topology
we refer to Ethier and
Kurtz \cite{MR838085}.
% \begin{equation}\label{eqn-X_T}
%  X_T:= L^\infty(0,T;\ve) \cap L^2(0,T;\be)
%  \end{equation}

% \del{ Let $(\Omega,
% \mathcal{F}, \mathbb{P})$ be a complete probability space equipped
% with a filtration $\mathbb{F}=\{\mathcal{F}_t: t\geq 0\}$
% satisfying the usual condition. Let us assume also that $\rK$ is a separable Hilbert space and $W=(W(t))_{t\geq 0}$ be a $\rK$-cylindrical
% Wiener process on $(\Omega,
% \mathcal{F},\mathbb{F}, \mathbb{P})$. }

 Let $(Z,\mathcal{Z})$ be a separable metric space and let
 $\nu$ be a $\sigma$-finite positive measure on it. Suppose that $\mathfrak{P}=(\Omega,\mathcal{F},\FF,\mathbb{P})$ is a filtered
probability space, where $\FF=(\mathcal{F}_t)_{t\geq 0}$ is a filtration,
and $\eta: \Omega\times \mathcal{B}(\mathbb{R}_+)\times\mathcal{Z}\to
\bar{\mathbb{N}}$ is a time homogeneous Poisson random measure with
the intensity measure $\nu$ defined over the filtered probability
space $\mathfrak{P}$. We will denote by $\tilde\eta=\eta-\gamma$
the compensated Poisson random measure associated to $\eta$ where
the compensator $\gamma$ is given by
$$
\mathcal{B}(\mathbb{R}_+)\times \mathcal{Z}\ni (I, A)\mapsto  \gamma(I,A)=\nu(A)\lambda(I)\in\mathbb{R} _ +.
$$
%This will be fixed in the whole section.
For each Banach space $B$ we denote by $M^2(0,T;B)$ the  space of all progressively measurable $B$-valued processes such that
$$\lve u\rve^2_{M^2(0,T;B)}=\mathbb{E}\int_0^T \lve u(s)\rve^2_{B}ds<\infty.$$
Throughout the paper, let us denote
by $M^2(X_T)$,
the space of all progressively measurable $\ve\cap \h$-valued processes
whose trajectories belong to $X_T$ almost surely, endowed with a
norm
\begin{equation}\label{eqn-M^2X_T}
\Vert u \Vert_{M^2(X_T)}^2 = \mathbb{E}\Big[ \sup_{s \in [0,T]}
\lve u(s)\rve^2+\int_0^T \Vert u(s) \Vert_\ast^2\, ds\Big].
\end{equation}
% It is then known, see e.g. \cite{Brz+Haus_2009}, \cite{Zhu_2010}, that there exists a unique
% continuous linear operator $I$ which associates to each progressively measurable
% process $\xi:\Rb{+}\times Z\times \Omega \to \ve$ such that
Let $H$ be a separable Hilbert space. {Following the notation of \cite{Brz+Haus_2009}, let
$\mathcal{M}^2(\Rb{+},L^2(Z,\nu,H))$ be the class of all progressively measurable processes
$\xi:\Rb{+}\times Z\times \Omega \to H$ satisfying the condition
\begin{equation}\label{cond-2.01}
          \mathbb{E} \int_0^{T} \int_Z\vert \xi(r,z)\vert^2_H\nu(dz)\,dr <\infty, \quad \forall T>0.
\end{equation}
If $T>0$, the class of all
progressively measurable processes $\xi:[0,T]\times  Z\times
\Omega \to H$ satisfying the condition \eqref{cond-2.01} just
for this one $T$,  will be denoted by
$\mathcal{M}^2(0,T,L^2(Z,\nu,H))$.
Also, let $\mathcal{M}_{step}^2(\Rb{+},L^2(Z,\nu,H))$ be the space of all processes $\xi \in \mathcal{M}^2(\Rb{+},L^2(Z,\nu,H))$ such that
$$
\xi(r) = \sum_{j=1} ^n 1_{(t_{j-1}, t_{j}]}(r) \xi_j,\quad 0\le r,
$$
where $\{0=t_0<t_1<\ldots<t_n<\infty\}$ is a partition of $[0,\infty)$, and
for all $j$,    $\xi_j$ is  an $\mathcal{F}_{t_{j-1}}$ measurable random variable. For any $\xi\in\mathcal{M}_{step}^2(\Rb{+},L^2(Z,\nu,H))$ we set
\begin{equation}
\label{eqn-2.02} \tilde{I}(\xi)= % \int_0^t \int_Z \xi(r,z) \tilde{\eta}(dz,dr)
\sum_{j=1}^n  \int_Z  \xi_j (z) \tilde\eta \left(dz,(t_{j-1}, t_{j}] \right).
\end{equation}
Basically, this is the definition of stochastic integral of a random step process $\xi$ with respect to the compound random Poisson measure $\tilde{\eta}$.
The extension of this integral on $\mathcal{M}^2(\Rb{+},L^2(Z,\nu,H))$ is possible thanks to the following result which is taken from
\cite[Theorem C.1]{Brz+Haus_2009}.
\begin{thm}\label{THM-ZB+EH}
There exists a unique bounded linear operator
$$ I  : \mathcal{M}^2 (\mathbb{R}_+, L^2 (Z, \nu; H)) \rightarrow L^2 (\Omega,\mathcal{F}; H)$$
such that for $\xi \in  \mathcal{M}_{step}^2(\mathbb{R}_+, L^2(Z, \nu; H))$ we have $I(\xi ) =  \tilde{I}(\xi)$. In particular, there exists a constant
$C=C(H)$ such that for any $\xi\in \mathcal{M}^2(\Rb{+},L^2(Z,\nu,H))$,
\begin{equation}
\label{ineq-2.03} \mathbb{E} \vert \int_0^t \int_Z \xi(r,z)
\tilde{\eta}(dz,dr)\vert ^2_H \leq C\, \mathbb{E}
\int_0^t\int_Z\vert \xi(r,z)\vert^2_H\, \nu(dz)\,dr, \; t\geq 0.
\end{equation}
Moreover, for each $\xi\in \mathcal{M}^2(\Rb{+},L^2(Z,\nu,H))$ , the process $I (1_{[0,t]} \xi)$, $t \ge 0$, is an
$H$-valued \cadlag martingale. The process $1_{[0,t]}\xi$ is defined by $[1_{[0,t]} \xi ](r, z, \omega)
:= 1_{[0,t]} (r)\xi(r, z, \omega),$ $t\ge 0$, $r \in \mathbb{R}_+$ , $z\in Z$ and $\omega\in \Omega$.

As usual we will write
$$
 \int_0^t \int_Z \xi(r,z)
\tilde{\eta}(dz,dr) := I(\xi)(t),\quad t\ge 0.
$$

\end{thm}}

Now we introduce the following standing assumptions.
\begin{assum}\label{lin-terms}
We will identify $\h$ with its dual $\h^\ast$, and \del{. We also identify $\h$ with the dual of $E$} we assume that the embeddings
$$\be \subset \ve \subset \h\subset \ve^\ast \subset \be^\ast $$ are continuous and dense.

 Let $N$ be a self-adjoint operator on $\h$ such that $N \in \mathcal{L}(\be, \h)\cap \mathcal{L}(\ve, \ve^\ast)$. Also let $A$ be a bounded linear map from $\be$ into $\h$. We assume that there exist $C_N, C_A >0$ such that
 $$\langle Au, Nu\rangle\ge   C_A \lve u\rve^2_\ast \text{ and }
  \langle Nu, u\rangle \ge C_N \lve u \rve^2,$$
for any $u\in V$. The norm of $N\in \mathcal{L}(\be,\h)$ and $N\in \mathcal{L}(\ve,\ve^\ast)$  will be denoted respectively by $ \Lve N\Lve_{\mathcal{L}(\be,\h)} $ and
 $\Lve N \Lve_{\mathcal{L}(\ve,\ve^\ast)}$ throughout.
\end{assum}

Let $F$ and $G$ be two nonlinear mappings satisfying the following
sets of conditions.
\begin{assum}\label{assum-F}
 Suppose that $F: \be \to \h$ is such that
$F(0)=0$ and there exists $p \geq 1$, $\alpha \in [0,1)$ and $C>0$
such that
\begin{equation}\label{eqn-local Lipschitz-F}
\lvert F(y)-F(x) \lvert \leq C \Big[ \lve y-x\lve \lve
y\lve^{p-\alpha} \Vert y\Vert_\ast^\alpha + \Vert y-x\Vert_\ast^\alpha
\lve y-x\lve^{1-\alpha} \lve x\lve^p\Big],
\end{equation}
for any $x, y\in \be$.
\end{assum}
\del{Given two Hilbert spaces $K$ and $H$,  we denote by $\mathcal{J}_2(K,H)$ the Hilbert space of all Hilbert-Schmidt operators from $K$ to $H$. For any
$G\in \mathcal{J}_2(K,H)$ its Hilbert-Schmidt norm will be denoted throughout the paper by $\lve G\rve_{\mathcal{J}_2}$.}
\begin{assum}\label{assum-G}
\begin{enumerate}[(i)]
 \item \label{part-i}
 Assume that $G: \ve \to L^{2p}(Z,\nu, \ve)$ and there exists a constant $\ell_p>0$ such that
\begin{equation}\label{eqn-local Lipschitz-G}
\lVert G(x)-G(y)\rVert^{2p}_{L^{2p}(Z,\nu, \ve)}\le \ell^p_p \rve x-y\rve^{2p},
\end{equation}
for any $x, y\in \ve$ and $p=1,2$.

Note that this implies in particular that there exists a constant $\tilde{\ell}_p>0$ such that
\begin{equation}\label{Lipschitz-G-2}
\lVert G(x)\rVert^{2p}_{L^{2p}(Z,\nu, \ve)}\le \tilde{\ell}^p_p(1+ \rve x\rve^{2p}),
\end{equation}
for any $x\in \ve$ and $p=1,2$.
\item \label{part-ii} We also assume that $G$ satisfies the inequality \eqref{eqn-local Lipschitz-F} with the norm of $\ve$ replaced by the norm of $\h$. More precisely, there exists
 $\ell_p>0$ such that
\begin{equation}\label{eqn-local Lipschitz-G-H}
\lVert G(x)-G(y)\rVert^{2p}_{L^{2p}(Z,\nu, \h)}\le \ell^p_p \rvert x-y\rvert^{2p},
\end{equation}
for any $x, y\in \ve$ and $p=1,2$.

\end{enumerate}

\end{assum}

Throughout this work we fix a positive number $T$. One of our objectives is to prove the existence and uniqueness of maximal/local solution of the following stochastic evolution equation
\begin{equation}\label{Full-EQ}
 \bu(t)= \bu_0-\int_0^{t} \Big[A\bu(s)+F(\bu(s))\Big]ds+\int_0^{t}\int_Z  G(z,\bu(s))\wie(dz,ds).
\end{equation}
The above identity is the shorthand of the following identity
\begin{equation}\label{WEAK-EQ}
 \langle \bu(t), v\rangle= \langle \bu_0, v\rangle-\int_0^{t} \langle \big[A\bu(s)+F(\bu(s))\big], v\rangle ds+\int_0^{t} \int_Z \langle  G(\bu(s)), v \rangle \wie(dz,ds),
\end{equation}
for any $t\in [0,T]$ and $v\in \h$.
%Before we introduce the assumption on problem \eqref{Full-EQ}

 Now, let us introduce the concept of local and maximal local solution.
%we give the concepts of solution in which we are interested in.
\begin{Def}[Local solution]\label{Def-Loc-SOL}
 By a local solution of \eqref{Full-EQ} we mean a pair $(\bu,\tau_\infty)$ such that
 \begin{enumerate}
  \item the symbol $\tau_\infty$ is a stopping time such that $\tau_\infty\le T $ a.s. and there exists a nondecreasing sequence $\{\tau_n, n\ge 1\}$
  stopping times with $\tau_n\uparrow \tau_\infty$ a.s. as $n\uparrow  \infty$,
  \item the symbol $\bu$ denotes a progressively measurable stochastic process such that $u\in X_t$ a.s. for any $t\in [0,\tau_\infty)$ and
  \begin{equation}\label{Full-EQ-2}
 \bu(t\wedge \tau_n)= \bu_0-\int_0^{t\wedge \tau_n} \Big[A\bu(s)+F(\bu(s))\Big]ds+\int_0^{t\wedge \tau_n} \int_Z  G(z,\bu(s))\wie(dz,ds),
\end{equation}
holds for any $t\in [0,T]$ and $n\ge 1$ with probability 1.
 \end{enumerate}
The identity \eqref{Full-EQ-2} is the shorthand of the following
\begin{equation}\label{WEAK-EQ-2}
 \langle \bu(t\wedge \tau_n), v\rangle= \langle \bu_0, v\rangle-\int_0^{t\wedge \tau_n} \langle \big[A\bu(s)+F(\bu(s))\big], v\rangle ds+\int_0^{t\wedge \tau_n}\int_Z \langle  G(\bu(s)), v \rangle \wie(dz,ds),
\end{equation}
holds for any $t\in [0,T]$, $v\in \h$, and $n\ge 1$ with probability 1.
\end{Def}
We also define the maximal local solution to \eqref{Full-EQ}.
\begin{Def}[Maximal local solution]\label{Def-Max-SOL}
\begin{enumerate}
 \item
 Let $(\bu,\tau_\infty)$ be a local solution to \eqref{Full-EQ} such that $\lim_{t\toup \tau_\infty}\lve \bu\rve_{X_t}=\infty$ on $\{\omega, \tau_\infty< T\}$,
then the local process $(\bu,\tau_\infty)$ is called a maximal local solution. If $\tau_\infty<T$, then the stopping time $\tau_\infty$ is called the explosion time of the stochastic process $u$.
\item A maximal local solution $(\bu,\tau_\infty)$ is said unique if for any other maximal local solution $(\v,\sigma_\infty)$ we have $\sigma_\infty=\tau_\infty$ and
$\bu(t)=\bv(t)$ for any $0\le t<\tau_\infty$ with probability one.
\item If the explosion time of the stochastic process $\bu$ is equal to $T$ with probability 1, then the stochastic process $\{\bu(t), t\in [0,T]\}$ is called a global solution.
\end{enumerate}
\end{Def}

As in \cite{SLC-JAP}, we let
$\theta:\mathbb{R}_+\to [0,1]$ be a ${\mathcal C}^\infty_0$ non
increasing function
 such that
\begin{equation}\label{eqn-theta} \inf_{x\in\mathbb{R}_+}\theta^\prime(x)\geq -1, \quad \theta(x)=1\;
\mbox{\rm  iff } x\in [0,1]\quad \mbox{\rm  and } \theta(x)=0 \;
\mbox{\rm  iff } x\in [2,\infty).
\end{equation}
and for $n\geq 1$ set  $\theta_n(\cdot)=\theta(\frac{\cdot}{n})$.
Note that if $h:\mathbb{R}_+\to\mathbb{R}_+$ is a non decreasing
function, then for every $x,y\in {\mathbb R}$,
\begin{equation}\label{ineq-theta}
 \theta_n(x)h(x) \leq h(2n),\quad
%\\\label{ineq-Lip-theta}
\vert \theta_n(x)-\theta_n(y)\vert \leq \frac1n |x-y| . %, \;\; \mbox{ for all } x_1,x_ 2\in\mathbb{R}.
\end{equation}
\begin{prop}\label{prop-global Lipschitz-F}
Let $F$ be a nonlinear mapping satisfying Assumption
\ref{assum-F}. Let us consider a map $B_n^T: X_T  \to L^2(0,T;\h) $ defined by
$$ B_n^T(u)(t):= \theta_n( \Vert u
\Vert_{X_t}) F(u(t)), \,\, u\in X_T,\,\, t\in[0, T].$$
Then $B_n^T$ is globally Lipschitz and moreover, for any
$u_1,u_2 \in X_T$,
\begin{equation}\label{eqn-global Lipschitz-F}
\Vert B_n^T(u_1)-B_n^T(u_2)\Vert_{L^2(0,T;\h)} \leq  C (2n)^{p}\Big[   (2n) C +1\Big]  T^{\frac{1-\alpha}2} \Vert  u_1 -u_2
\Vert_{X_T}.
\end{equation}
%where $\alpha=\lambda_1+\lambda_2$.

\end{prop}

% The proof is based on a proof from the paper
% \cite{Brz+Millet_2012} which in turn was based on a proof from
% papers by De Bouard and Debussche \cite{deBouard+Deb_1999,deBouard+Deb_2003}.
\begin{proof}
The proof is the same as in \cite{SLC-JAP}, but we repeat it here for sake of completeness. Note that by Assumption \ref{assum-F} $B_n^T(0)=0$. Assume that $u_1,u_2 \in X_T$. Denote,
for $i=1,2$,
\[
\tau_i= \inf\{t \in [0,T]: \Vert u_i \Vert_{X_t} \geq 2n\}.
\]
Note that by definition, if the set on the RHS above is empty,
then $\tau_i=T$.  Without loss of generality we may assume that $\tau_1\le \tau_2.$

We have the following chain of inequalities/equalities
\begin{eqnarray*}
\Vert B_n^T(u_1)-B_n^T(u_2)\Vert_{L^2(0,T;\h)} &=& \Big[ \int_0^T \vert  \theta_n( \Vert u_1 \Vert_{X_t}) F(u_1(t))-\theta_n( \Vert u_2 \Vert_{X_t}) F(u_2(t))\vert^2\,dt\Big]^{1/2}\\
&&\mbox{ because for } i=1,2, \;\; \theta_n( \Vert u_i
\Vert_{X_t}) =0 \mbox{ for } t \geq \tau_2
\\
&=& \Big[ \int_0^{\tau_2} \vert  \theta_n( \Vert u_1 \Vert_{X_t})
F(u_1(t))-\theta_n( \Vert u_2 \Vert_{X_t})
F(u_2(t))\vert^2\,dt\Big]^{1/2}
\\
&&\hspace{-5truecm}\lefteqn{= \Big[ \int_0^{\tau_2} \vert \big[
\theta_n( \Vert u_1 \Vert_{X_t}) - \theta_n( \Vert u_2
\Vert_{X_t}) \big] F(u_2(t))+ \theta_n( \Vert u_1
\Vert_{X_t})\big[  F(u_1(t))- F(u_2(t)) \Big]
\vert^2\,dt\Big]^{1/2} }
\\
& &\hspace{-3truecm}\lefteqn{ \leq
\Big[ \int_0^{\tau_2} \vert  \big[ \theta_n( \Vert u_1 \Vert_{X_t}) - \theta_n( \Vert u_2 \Vert_{X_t}) \big] F(u_2(t))\vert^2 \,dt \Big]^{1/2} }\\
&&\hspace{-2truecm}\lefteqn{+ \Big[ \int_0^{\tau_2} \vert
\theta_n( \Vert u_1 \Vert_{X_t})\big[  F(u_1(t))- F(u_2(t)) \big]
\vert^2\,dt\Big]^{1/2} =:I_1+I_2.}
\end{eqnarray*}
Next, since $\theta_n$ is Lipschitz with Lipschitz constant $n^{-1}$
we have
\begin{eqnarray*}
I_1^2&=& \int_0^{\tau_2} \vert  \big[ \theta_n( \Vert u_1 \Vert_{X_t}) - \theta_n( \Vert u_2 \Vert_{X_t}) \big] F(u_2(t))\vert^2\,dt\\
&\leq& n^{-2 }C^2 \int_0^{\tau_2} \big[ \vert
\Vert u_1 \Vert_{X_t} -  \Vert u_2 \Vert_{X_t} \vert \big]^2 \vert F(u_2(t))\vert^2 \,dt\\
&&\mbox{by Minkowski inequality}\\
&\leq& n^{-2} C^2 \int_0^{\tau_2}  \Vert
 u_1 -u_2 \Vert_{X_t}^2 \vert F(u_2(t))\vert^2 \,dt \leq  4n^2 C^2 \int_0^{\tau_2}  \Vert
 u_1 -u_2 \Vert_{X_T}^2 \vert F(u_2(t))\vert^2 \,dt \\
&\leq & n^{-2} C^2 \Vert
 u_1 -u_2 \Vert_{X_T}^2 \int_0^{\tau_2}   \vert F(u_2(t))\vert^2 \,dt.
\end{eqnarray*}
%%%%%%%%%%%%%%%%%%%%%%%%%%%%%%%%%%%%%%
%%%%%%%%%%%%%%%%%%%%%%%%%%%%%%%%%%%%%%%%%%%%%%%%%%%%
%%%%%%%%%%%%%%%%%%%%%%%%%%%%%%%%%%%%%%%%%%%%%%%%%%%%%%%%%%%%%%%%%
%%%%%%%%%%%%%%%%%%%%%%%%%%%%%%%%%%%%%%%%%%%%%%%%%%%%%%%%%%%%%%%%%
Next, by assumptions
\begin{eqnarray*}
\int_0^{\tau_2}   \lvert F(u_2(t))\lvert^2 \,dt  &\leq & C^2  \int_0^{\tau_2}  \lve  u(t)\lve^{2p+2-2\alpha}  \lVert  u(t)\lVert_\ast^{2\alpha} \,dt\\
&\leq & C^2 \sup_{t \in [0,\tau_2]} \lve u(t)\lve^{2p+2-2\alpha}
\big( \int_0^{\tau_2} \lVert u(t)\lVert_\ast^{2} \,dt\big)^\alpha
\tau_2^{1-\alpha}
\\ &\leq & C^2  \tau_2^{1-\alpha} \lVert u \lVert_{X_{\tau_2}}^{2p+2} \leq  C^2  \tau_2^{1-\alpha}(2n)^{2p+2}.
\end{eqnarray*}
Therefore,
\begin{eqnarray*}
I_1&\leq & C^2   \tau_2^{(1-\alpha)/2}(2n)^{p} \lVert  u_1 -u_2
\lVert_{X_T}.
\end{eqnarray*}

Also, because $ \theta_n( \lVert u_1 \lVert_{X_t}) =0$ for  $ t \geq
\tau_1$, and $\tau_1 \leq \tau_2$, we have
\begin{eqnarray*}
I_2&=& \Big[ \int_0^{\tau_2} \lvert  \theta_n( \lVert u_1 \lVert_{X_t})\big[  F(u_1(t))- F(u_2(t)) \big] \lvert^2\,dt \Big]^{1/2}\\
&=& \Big[ \int_0^{\tau_1} \lvert   \theta_n( \lVert u_1 \lVert_{X_t})\big[  F(u_1(t))- F(u_2(t)) \big] \lvert^2\,dt\Big]^{1/2}\\
&&\mbox{because } \theta_n( \lVert u_1 \lVert_{X_t}) \leq 1 \mbox{ for } t\in [0,\tau_1) \\
&\leq& \Big[ \int_0^{\tau_1} \lvert     F(u_1(t))- F(u_2(t))  \lvert^2\,dt\Big]^{1/2}\\
&\leq & C \Big[ \int_0^{\tau_1}     \lve u_1(t)-u_2(t)\lve^2 \lve u_1(t)\lve^{2p-2\alpha} \lVert u_1(t)\lVert_\ast^{2\alpha}   \,dt\Big]^{1/2}\\
&+& C \Big[ \int_0^{\tau_1}  \lVert u_1(t)-u_2(t)\lVert_\ast^{2\alpha} \lve u_1(t)-u_2(t)\lve^{2-2\alpha} \lve u_2(t)\lve^{2p} \,dt\Big]^{1/2} \\
&\leq&
C \sup_{t\in [0,\tau_1]} \lve u_1(t)-u_2(t)\lve^{1-\alpha}
\lve u_2(t)\lve^p \Big[ \int_0^{\tau_1} \lVert
u_1(t)-u_2(t)\lVert_\ast^{2\alpha} \,dt\Big]^{1/2}
\\
&+& C \sup_{t\in [0,\tau_1]} \lve u_1(t)-u_2(t)\lve \lve
u_1(t)\lve^{p-\alpha} \Big[ \int_0^{\tau_1} \lVert
u_1(t)\lVert_\ast^{2\alpha}   \,dt\Big]^{1/2}\\
&\leq&  C \sup_{t\in [0,T]} \lve u_1(t)-u_2(t)\lve \sup_{t\in
[0,\tau_1]} \lve u_1(t)\lve^{p-\alpha} \Big[ \int_0^{\tau_1}
\lVert u_1(t)\lVert_\ast^{2}   \,dt\Big]^{\alpha/2}
\tau_1^{(1-\alpha)/2}
\\
&+& C \sup_{t\in [0,T]} \lve u_1(t)-u_2(t)\lve^{1-\alpha}
\sup_{t\in [0,\tau_1]}  \lve u_2(t)\lve^p \Big[ \int_0^{\tau_1}
\lVert u_1(t)-u_2(t)\lVert_\ast^{2}   \,dt\Big]^{\alpha/2}
\tau_1^{(1-\alpha)/2}
\\
&\leq&  C \lVert u_1-u_2\lVert_{X_T}  \lVert u_1 \lVert_{X_{\tau_1}}^p \tau_1^{(1-\alpha)/2} +C \lVert u_1-u_2\lVert_{X_T}  \lVert u_2 \lVert_{X_{\tau_1}}^p \tau_1^{(1-\alpha)/2}\\
&&\mbox{ because} \lVert u_1 \lVert_{X_{\tau_1}} \leq 2n \mbox{ and } \lVert u_2 \lVert_{X_{\tau_1}} \leq \lVert u_2 \lVert_{X_{\tau_2}}\leq 2n\\
&\leq&   C \tau_1^{(1-\alpha)/2} \lVert u_1-u_2\lVert_{X_T} \Big[
\lVert u_1 \lVert_{X_{\tau_1}}^p  +   \lVert u_2
\lVert_{X_{\tau_1}}^p\Big] \leq C (2n)^{p+1} \tau_1^{(1-\alpha)/2}
\lVert u_1-u_2\lVert_{X_T}
\end{eqnarray*}
Summing up, we proved
\begin{eqnarray*}
\Vert B_n^T(u_1)-B_n^T(u_2)\Vert_{L^2(0,T;\h)} &\leq &\Big[ C^2
\tau_2^{(1-\alpha)/2}(2n)^{p}
+C (2n)^{p+1} \tau_1^{(1-\alpha)/2}\Big] \Vert u_1-u_2\Vert_{X_T}\\
&=& C (2n)^{p}\big[   2n C +1\Big]  \tau_2^{(1-\alpha)/2} \Vert
u_1 -u_2 \Vert_{X_T}
\end{eqnarray*}
The proof is complete.

\end{proof}

\section{Existence of Maximal local and Global solution of Eq. \eqref{Full-EQ} }\label{MAX-GLO}
This section is devoted to the solvability of \eqref{Full-EQ}. We will mainly show that under  Assumption \ref{lin-terms}-\ref{assum-G},  Eq. \eqref{Full-EQ} admits a unique maximal local solution.
 Under additional conditions on $F$ and $G$ we prove that this maximal local solution turns out to be a global solution.
 The results are obtained by use of cut-off and fixed point arguments.

 \subsection{Global solution of a truncated equation}
%This subsection is about the following equation
For simplicity we set $B_n^T(u)(s)=B_n^T(u(s))$ for any $u\in X_T$ and $s\ge 0$. Let
\begin{equation}\label{Truncated-Eq}
  \bu_n(t)+\int_0^t [A\bu_n(s)+B_n^T(\bu_n(s))]ds=\bu_0 +\int_0^t \int_Z G(z,\bu_n(s))\wie(dz,ds), \quad t\in [0,T],
 \end{equation}
which is understood as
\begin{equation}\label{Truncated-Eq-2}
  \langle \bu_n(t), v\rangle +\int_0^t \langle A\bu_n(s)+B_n^T(\bu_n(s)), v\rangle ds=\langle \bu_0, v \rangle +\int_0^t \int_Z \langle G(z,\bu_n(s)), v\rangle
  \wie(dz,ds), \quad t\in [0,T],
 \end{equation}
for any $v\in \h$. Here,  we previously set $$B_n^T(u(t))=\theta_n(\lVert u\rVert_{X_t})F(u(t)),$$ for any $u\in X_t$ and $t\ge0$.
For $n \in \mathbb{N}$ we also set
\begin{equation}\label{phin}
 \phi(n)=C^2 (2n)^{2p}\Big[   2n C +1\Big]^2 .
\end{equation}
Now, let $\bv \in M^2(X_T)$, $n>0$ and let us consider the linear stochastic evolution equation
\begin{equation}\label{lin-eq-1}
\begin{cases}
 d\bu_n(t) + A\bu_n(t)dt=-B_n^t(\bv(t))dt+\int_Z G(z,\bv(t)) \wie(dz,dt),\\
 \bu_n(0)=u_0.
\end{cases}
\end{equation}
Thanks to Theorem \ref{Appendix-LIN} for each  $\bv \in M^2(X_T)$ and $n\ge 1$, there exists a unique $\ve$-valued progressively measurable  process
$\bu_n$ solving \eqref{lin-eq-1}. Moreover,
$\bu^n \in \mathbf{D}(0,T;\ve)\cap
L^2(0,T; \be)$ with probability 1.
 %We shall prove that $\Lambda_n$ maps $M^2(X_T)$ into itself and it is a strict contraction for small $T>0$.
\begin{lem}\label{Lambda-ITS}
For each $n\ge 1$ let $\Lambda_n$ be the mapping defined by $$\Lambda_n: M^2(X_T)\ni \bv
\mapsto \bu_n=\Lambda_n(\bv),$$ where $\bu_n$ is the unique solution to \eqref{lin-eq-1}.
 For any $\bv\in M^2(X_T)$, the stochastic process $\bu_n$ belongs to $M^2(X_T)$.
\end{lem}
\begin{proof}
\del{ To ease notation we will omit the dependence on $n$.}
Let $\Psi: \h \to \mathbb{R}$ be the
mapping defined by $$\Psi(u)=\langle u, Nu  \rangle,$$ for any $u\in \h$. This mapping is Fr\'echet differentiable with first derivative defined by
\begin{equation*}
\Psi^\prime(u)[h]=\langle h, N u\rangle+ \langle u,
Nh\rangle.
\end{equation*}
Since $N$ is self-adjoint we have
\begin{equation*}
\Psi^\prime(u)[h]=2\langle h, Nu\rangle.
\end{equation*}

Applying It\^o's formula (see, for instance, \cite[Appendix D]{Pes+Zab}) to $\Psi(\bu)$ with \eqref{lin-eq-1} we obtain
\begin{equation}\label{ITO-1}
\begin{split}
 & \Psi(\bu_n(t))-\Psi(\bu_0) +2\int_0^t\langle A\bu_n(s)+ B_n^T(\bv(s)), N\bu_n(s)\rangle ds\\
 & \quad \quad = \int_0^t \int_Z \biggl[\Psi(\bu_n(s-)+G(z,\bv(s)))-\Psi(\bu_n(s-))-\Psi^\prime(\bu_n(s-))[G(z,\bv(s))] \biggr]\nu(dz)ds\\
 & \quad \quad \quad + \int_0^t \int_Z\biggl[\Psi(\bu_n(s-)+G(z,\bv(s))-\Psi(\bu_n(s-))\biggr]\wie(dz,ds).
 \end{split}
\end{equation}
From the Cauchy-Schwarz inequality we derive that
\begin{align*}
 \biggl\lvert \int_0^t \langle B_n^T(\bv(s)), N\bu_n(s) \rangle ds\biggr\lvert \le & \int_0^t \lvert B_n^T(\bv(s))\rvert \lvert N\bu_n(s) \rvert ds, \\
 \le & \noe \int_0^t \lvert B_n^T(\bv(s))\rvert \lve \bu_n(s)\rve_\ast ds.
 %\le & C_\eps C(n) t^{(1-\alpha)/2} \lve v\rve^2_{L^(0,T,)}+ \eps  t^{(1-\alpha)/2}\int_0^t \lvert u \rvert^2 ds.
\end{align*}
From the last line along with Cauchy's inequality with $\varepsilon$ we deduce that
\begin{equation*}
 \me \biggl\lvert \int_0^t \langle B_n^T(\bv(s)), N\bu_n(s) \rangle ds\biggr\lvert \le \eps \me \int_0^t\Vert \bu_n(s)\Vert_\ast^2 ds+
 \frac{\noe^2}{4\eps} \me\int_0^t \lvert B_n^T(\bv(s))\rvert^2 ds.
\end{equation*}
Now invoking Eq. \eqref{eqn-global Lipschitz-F} from Proposition \ref{prop-global Lipschitz-F} we infer that
\begin{equation}\label{ineq-NL-te}
 \me \biggl\lvert \int_0^t \langle B_n^T(\bv(s)), N\bu_n(s) \rangle ds\biggr\lvert \le \eps \me \int_0^t\Vert \bu_n(s)\Vert^2_\ast ds+
 \frac{\noe^2 \phi(n)}{4\eps} t^{\alpha-1} \lve \bv\rve^2_{M^2(X_T)},
\end{equation}
where $\phi(n)$ is defined in \eqref{phin}.

Now, note that
$$ \Psi(u+h)-\Psi(u)-\Psi^\prime(u)[h]=\langle Nh, h \rangle.$$
Hence
\begin{align}
I_1:=  \me \biggl \lvert \int_0^t \int_Z \biggl[\Psi(\bu_n(s-)+G(z,\bv(s)))-\Psi(\bu_n(s-))-\Psi^\prime(\bu_n(s-))[G(z,\bv(s))]
\biggr]\nu(dz)ds\biggr\rvert
\nonumber \\
\le \nov \EE
\int_0^t \int_Z \lVert G(z,\bv(s))\rVert^2 \nu(dz) ds.
\end{align}
By making use of \eqref{Lipschitz-G-2} we easily derive from the last inequality that
\begin{equation}
 I_1 \le  t \nov \tilde{\ell}_1 (1+\lve \bv\rve^2_{M^2(X_T)}). \label{ineq-ST-co}
\end{equation}

Notice also that
$$\Psi(u+h)-\Psi(u)= 2\langle Nu, h\rangle + \langle Nh, h\rangle,$$
thus
\begin{equation*}
 \begin{split}
  I_2:= &\EE \sup_{s\in [0,t]}\biggl\lvert \int_0^s \int_Z\biggl[\Psi(\bu_n(r-)+G(z,\bv(r))-\Psi(\bu_n(r-))\biggr]\wie(dz,dr) \biggr\lvert
  \\
  & \le
  \EE \sup_{s\in [0,t]}\biggl\lvert \int_0^s \int_Z \langle N\bu_n(s-), G(z,\bv(s))\rangle \wie(dz,ds)\biggr\lvert\\
  & \quad \quad  + \EE \sup_{s\in [0,t]}\biggl\lvert \int_0^s \int_Z \langle NG(z,\bv(s)), G(z,\bv(s))\rangle \wie(dz,ds)\biggr\lvert\\
  & \le I_{2,1}+I_{2,2}.
 \end{split}
\end{equation*}
Owing to the BDG inequality (see, for instance, \cite[Theorem 48]{Protter}) we infer that
\begin{align}
 I_{2,1}:=& \EE \sup_{s\in [0,t]}\biggl\lvert \int_0^s \int_Z \langle N\bu_n(r-), G(z,\bv(r))\rangle \wie(dz,dr)\biggr\lvert\nonumber \\
 & \le
  C \me \biggl[\int_0^t \int_Z \langle N G(z,\bv(s)), \bu_n(s-) \rangle^2 ds \biggr]^\frac12,\nonumber\\
 & \le  C \nov \me\biggl[\sup_{s\in [0,t]}\lve \bu_n(s)\rve \biggl(\int_0^t \int_Z \lve G (z,\bv(s))\rve ^2 \nu(dz) ds\biggr)^\frac12\biggr]\nonumber\\
 & \mbox{(by the Young inequality with $\delta>0$ arbitrary) }\nonumber\\
 \le & \delta \me\biggl[\sup_{s\in [0,t]}\lve\bu_n(s)\rve^2\biggr]+\frac{C^2 \nov^2 }{4\delta}
 \me \int_0^t \int_Z \lve G (z,\bv(s))\rve^2 \nu(dz) ds\nonumber\\
 & \mbox{(by the inequality \eqref{Lipschitz-G-2})}\nonumber\\
 I_{2,1}\le & \delta \me\biggl[\sup_{s\in [0,t]}\lve \bu_n(s)\rve^2\biggr]+\frac{C^2 \nov^2 \tilde{\ell}_1 t}{4\delta}(1+ \lve \bv\rve^2_{M^2(X_T)}).\label{ineq-ST-te}
\end{align}
Using again the BDG inequality yields
\begin{align}
 I_{2,2}& := \EE \sup_{s\in [0,t]}\biggl\lvert \int_0^s \int_Z \langle NG(z,\bv(s)), G(z,\bv(s))\rangle \wie(dz,ds)\biggr\lvert\nonumber\\
 & \le C  \me \biggl[\int_0^t \int_Z [\langle N G(z,\bv(s)), G(z,\bv(s)\rangle]^2\nu(dz) ds \biggr]^\frac12\nonumber\\
 & \le C \nov \Lve \me\biggl[\int_0^t\int_Z \lve G(z,\bv(s))\rve^4 \nu(dz) ds\Biggr]^\frac12\nonumber \\
 & \mbox{(by the inequality \eqref{Lipschitz-G-2})}\nonumber\\
 & \le C \nov \tilde{\ell}_2 t (1+\me\sup_{s\in [0,t]}\lve \bv(s)\rve^2),\nonumber\\
 I_{2,2} &\le C \nov \tilde{\ell}_2 t (1+\Vert \bv\Vert^2_{M^2(X_T)}).\label{ineq-ST-te-2}
\end{align}
Now it follows from Eqs. \eqref{ITO-1}, \eqref{ineq-NL-te}, \eqref{ineq-ST-co}, \eqref{ineq-ST-te} and \eqref{ineq-ST-te-2} that
\begin{equation*}
\begin{split}
  \me \sup_{s\in [0,t]} \Psi(\bu_n(s))-\Psi(\bu_0) +2\me \int_0^t \langle A \bu_n(s), N \bu_n(s)\rangle ds\le
 2C(\Lve N\Lve,\delta,\eps ,n, t) (1+\lve \bv\rve^2_{M^2(X_T)})\\
 + \eps \me \int_0^t \lve \bu_n(s)\rve_\ast^2 ds +\delta \me\biggl[\sup_{s\in [0,t]}\lve \bu_n(s)\rve^2\biggr],
\end{split}
\end{equation*}
where
\begin{align*}
\Lve N\Lve&:=\max\left(\noe, \nov\right),\\
C(\Lve N\Lve,\eps, \delta,n, t)&:= \biggl( \frac{\Lve N\Lve \phi(n) [t^{\alpha-1}}{4\eps} +t \left[\frac{C^2 \Lve N \Lve \tilde{\ell}_1]}{4\delta}+\tilde{\ell}_1
 +C \tilde{\ell}_2\right]  \biggr)\Lve N\Lve .
 \end{align*}
Since $\langle u, N u\rangle \ge C_N \lve u \rve^2$ and $\langle A u, Nu \rangle \ge C_A\lve u \rve_\ast^2$, it follows that
\begin{equation*}
\begin{split}
 (C_N-\delta) \me\biggl[\sup_{s\in [0,t]}\lve \bu_n(s)\rve^2\biggr]+(2C_A-\eps) \me \int_0^t \lve \bu_n(s)\rve_\ast^2 ds\le
 C(\Lve N\Lve,\eps,\delta, n, t) (1+\lve \bv\rve^2_{M^2(X_T)})
 \\ + \Psi(\bu_0).
 \end{split}
\end{equation*}
Choosing $\eps=C_A$ and $\delta=C_N/2$, we derive from the last inequality that
\begin{equation*}
 \me\biggl[\sup_{s\in [0,t]}\lve \bu_n(s)\rve^2\biggr]+\me \int_0^t \lve \bu_n(s)\rve_\ast^2 ds \le  \frac{\Psi(\bu_0)}{\min(C_N/2, C_A )}
+ \frac{C(\Lve N\Lve,C_A,C_N,n, t)}{\min(C_N/2, C_A )} (1+\lve \bv\rve^2_{M^2(X_T)}).
\end{equation*}
With this last inequality  we easily conclude the proof of the claim.
\end{proof}
\begin{lem}\label{LEM-contraction}
%  The mapping  $\Lambda_n:= \Lambda_{n,\bu_0}^T: \bv \in M^2(X_T) \mapsto \Lambda_n(\bv)=\bu_n$, where $\bu_n$ is the unique solution to \eqref{lin-eq-1},
%  is a strict contraction for small $T$.
Let $\Lambda_n$ be the mapping defined in Lemma \ref{Lambda-ITS} and  $$\Lve N\Lve:=\max\left(\noe, \nov\right).$$ Then, there exists a constant $\kappa>0$ depending only on $\Lve N\Lve$, $n$ and the constants in Assumptions \ref{lin-terms}-\ref{assum-G} such that
 \begin{equation*}
 \lve \Lambda_n(\bv_1)-\Lambda_n(\bv_2)\rve^2_{M^2(X_T)}\le \kappa \big[T^{\alpha-1}\vee T\big]\lve \bv_1-\v_2 \rve^2_{M^2(X_T)},
\end{equation*}
for any $\bv_1,\bv_2\in M^2(X_T)$.
%Hence for $\big[T^{\alpha-1}\vee T\big] < \kappa^{-1 }$ the mapping $\Lambda_n$ is a strict contraction.
\end{lem}
\begin{proof}
 Let $\bv_i$, $i=1,2$, be two elements of $M^2(X_T)$. To each $\bv_i$ one can associate a unique element $\bu_i\in M^2(X_T)$  which is a solution to Eq.
 \eqref{lin-eq-1} with the stochastic perturbation
  $B^t_n(\bv_i(t)) dt+ \int_Z G(z,\bv_i(t)) \wie(dz,dt)$ and initial condition $\bu_0$. In this proof we suppress the dependence on $n$ of the solution to
  \eqref{lin-eq-1}. The difference $\bu=\bu_1-\bu_2$ solves the linear equation
  \begin{equation}\label{difference}
  \begin{cases}
   d\bu(t) +A\bu(t) dt =[B_n^t(\bv_2(t))-B_n^t(\bv_1(t))]dt+\int_Z [G (z,z,\bv_1(t))-G(z,\bv_2(t))]\wie(dz,dt),\\
   \bu(0)=0.
   \end{cases}
  \end{equation}
  To simplify our notation we also set $\bv=\bv_1-\bv_2$.

As before we apply It\^o's formula (see, for instance, \cite[Appendix D]{Pes+Zab}) to $\Psi(u)=\langle Nu, u\rangle$ with \eqref{difference}. We then obtain
\begin{equation}\label{ITO-2}
\begin{split}
 \Psi(\bu(t))+2\int_0^t \langle A\bu(s), N\bu(s)\rangle ds\le 2 \int_0^t \lvert B_n^T(\bv_1(s))-B_n^T(\bv_2(s))\rvert \lvert N\bu(s)\rvert ds \\
 +\int_0^t \int_Z f(z,s,\bv_1,\bv_2) \nu(dz) ds\\+\int_0^t \int_Z g(z,s,\bv_1,\bv_2) \wie(dz,ds),
\end{split}
\end{equation}
with
\begin{equation*}
\begin{split}
 g(z,s,\bv_1,\bv_2):=\langle N[G (z,\bv_1(s))-G (z,\bv_2(s))], G (z,\bv_1(s))-G (z,\bv_2(s))\rangle\\
 +2 \langle [G (z,\bv_1(s))-G (z,\bv_2(s))], N\bu(s-)\rangle,
\end{split}
\end{equation*}
and
\begin{equation*}
 f(z,s,\bv_1,\bv_2):= \langle N[G (z,\bv_1(s))-G (z,\bv_2(s))], G (z,\bv_1(s))-G (z,\bv_2(s))\rangle.
\end{equation*}
Arguing as in the proofs of Eq. \eqref{ineq-NL-te}, \eqref{ineq-ST-co}, \eqref{ineq-ST-te} and \eqref{ineq-ST-te-2}, respectively, we obtain the following inequalities
\begin{equation*}
\begin{split}
\me  \int_0^t\lvert B_n^T(\bv_1(s))-B_n^T(\bv_2(s))\rvert \lvert N\bu(s)\rvert ds \le
\frac{\Lve N\Lve^2\phi(n)}{4\eps} t^{\alpha-1} \lve \bv\rve^2_{M^2(X_T)}\\
+\eps \me \int_0^t \lVert \bu(s)\rVert_\ast^2 ds,
\end{split}
\end{equation*}
\begin{equation*}
 \begin{split}
  \me \sup_{s\in [0,t]}\biggl\lvert \int_0^t \int_Z g(z,s,\bv_1,\bv_2) \wie(dz,ds)\biggr\rvert\le
   \Big[\frac{C^2 \Lve N \Lve^2 {\ell}_1}{4 \delta}+ \Lve N\Lve \ell_2 \Big] t \lve \bv\rve^2_{M^2(X_T)}\\
  +\delta \me\biggl[\sup_{s\in [0,t]}\lve \bu(s)\rve^2\biggr],
 \end{split}
\end{equation*}
\begin{equation*}
 \me \int_0^t \int_Z f(z,s,\bv_1,\bv_2) \nu(dz) ds \le \Lve N\Lve \ell_1 t \lve \bv\rve^2_{M^2(X_T)},
\end{equation*}
where $\eps, \delta$ are arbitrary positive numbers. By setting $T^\ast=
T \vee T^{\alpha-1} $ and
$$ \tilde{\kappa}:= \biggl(\Lve N\Lve \big[\frac{\phi(n)}{4\eps} + \frac{C^2 \ell_1}{4\delta}\big] +\ell_1
 +C \ell_2\biggr)\Lve N\Lve,$$  it follows from these inequalities and Eq. \eqref{ITO-2} that
\begin{equation*}
\begin{split}
 (C_N-\delta) \me\biggl[\sup_{s\in [0,t]}\lve \bu(s)\rve^2\biggr]+(2C_A-\eps) \me \int_0^t \lve \bu(s)\rve_\ast^2 ds\le
 \tilde{\kappa} T^\ast \lve \bv\rve^2_{M^2(X_T)},
 \end{split}
\end{equation*}
where we have used the fact that $\langle u, N u\rangle \ge C_N \lve u \rve^2$ and $\langle A u, Nu \rangle \ge C_A\lve u \rve_\ast^2$.
By choosing $\delta= C_N/2$ and $\eps=C_A$ we get from the last estimate that
\begin{equation*}
 \me\biggl[\sup_{s\in [0,t]}\lve\bu(s)\rve^2\biggr]+\me \int_0^t \lve \bu(s)\rve_\ast^2 ds \le \kappa T^\ast \lve v\rve^2_{M^2(X_T)},
\end{equation*}
where $\kappa:=\tilde{\kappa}/\min(C_N/2, C_A )$.
The last estimate means that
\begin{equation*}
 \lve \Lambda_n(\bv_1)-\Lambda_n(\bv_2)\rve^2_{M^2(X_T)}\le \kappa T^\ast \lve \bv_1-\bv_2 \rve^2_{M^2(X_T)}.
\end{equation*}
%Hence for $T^\ast < \kappa^{-1} $ the mapping $\Lambda_n $ is a strict contraction.
This completes the proof of our lemma.
\end{proof}
%%%%%%%%%%%%_%%%%%%%%%%%%%%%%%%%%%%%%%%%%%%%%%%%%
%%%%%%%%%%%%%%%%%%%%%%%%%%%%%%%%%%%%%%%%%%%%
%%%%%%%%%%%%%%%%%%%%%%%%%%%%%%%%%%%%%%%%%%%%%
Let $n$ be a fixed positive integer. It follows from Lemma \ref{Lambda-ITS} that $\Lambda^n_{T,\bu_0}:=\Lambda_n$ maps
$M^2(X_T)$ into itself. From the proof of Lemma  \ref{LEM-contraction} we deduce that $\Lambda^n_{T,\bu_0}$ is globally Lipschitz. Moreover it is a strict contraction for small $T$.
 Therefore we can find a time $\delta_n>0$ that is independent of the initial condition $\bu_0$ such that $\Lambda^n_{\delta,\bu_0}$ is $\frac 12 $-contraction.
 Hence it admits a unique fixed point $\bu_{n,\delta_n}\in M^2(X_{\delta_n})$
 which solves on the small interval $[0,\delta_n]$  the nonlinear stochastic evolution equation
 \begin{equation}\label{Truncated-Eq-1}
  \bu(t)+\int_0^t [A\bu(s)+B_n^T(\bu(s))]ds=\bu_0 +\int_0^t \int_Z G(z,\bu(s))\wie(dz,s), \quad t\in [0,\delta_n).
 \end{equation}
\del{which is understood as an equality of random variables with values in $\h$.}
\begin{lem}\label{Weak-continuity}
Let $\bu_{n,\delta_n}$ be a solution of \eqref{Truncated-Eq-1}. Then $\mathbb{P}$-almost surely $\bu_{n,\delta_n}: [0,\delta_n) \to \ve$ is \cadlag.
\end{lem}
\begin{proof}
For sake of simplicity we just write $\delta:=\delta_n$. Since the solution $\bu_{n,\delta}$ to the truncated equation \eqref{Truncated-Eq} belongs to $M^2(X_\delta)$, from Proposition
\ref{prop-global Lipschitz-F} and the fact that $A\in \mathcal{L}(\be, \h)$ we infer that $ A\bu_{n,\delta}(\cdot) + B_n^T(\bu_{n,\delta}(\cdot))$
is an element of $M^2(0,\delta; \h)$. From Theorem \ref{THM-ZB+EH} we derive that the process
 $\int_0^{\cdot} \int_Z G(\bu_{n,\delta}(s)) \wie(dz,s)$ belongs to $L^2(\Omega, \mathbf{D}(0,\delta;\ve))$ and define an $\mathbb{F}$-martingale.
Since $\mathbb{P}$-a.s
 \begin{equation*}
  \bu_{n,\delta}(t)+\int_0^t [A\bu_{n,\delta}(s)+B_n^T(\bu_{n,\delta}(s))]ds=u_0 +\int_0^t \int_Z G(z,\bu_{n,\delta}(s))\wie(dz,ds),
 \end{equation*}
 $t\in (0,\delta]$, it follows from the above remarks and \cite[Theorem 2]{Gyongy+Krylov} that $\mathbb{P}$-a.s. $\bu_{n,\delta}\in \mathbf{D}(0,\delta;
 \ve)$.
 %%%%%%%%%%%%%%%%%%%%%%%%%%%%%%%%%%%%%%%
 %%%%%%%%%%%%%%%%%%%%%%%%%%%%%%%%%%%%%%%
 %Since $\mathbb{P}$-a.s. $\bu_{n,\delta}\in C(0,\delta;\h)\cap
%  L^\infty(0,\delta;H)$ it follows from \cite[Theorem 1.2]{Strauss} that $\mathbb{P}$-a.s.
%  $\bu_{n,\delta}\in C(0,\delta;H_w)$ where $C(0,T, H_w)$ denotes the space of weakly continuous functions $u: [0,T]\to H$. Investigating closely
%  the proof of \cite[Theorem 1.2]{Strauss}we infer from \cite[Eq. (2.1), page 544]{Strauss} that $\mathbb{P}$-a.s. $\bu_{n,\delta}\in H$ for all $t\in [0,\delta]$.
%%%%%%%%%%%%%%%%%%%%%%%%%%%%%%%%%%%%%%%%%%%%%%%%
%%%%%%%%%%%%%%%%%%%%%%%%%%%%%%%%%%%%%%%%%%%%%%
\end{proof}
Now, we are able to formulate the result about the global existence of solution to the truncated equation \eqref{Truncated-Eq}.
\begin{thm}\label{Global-Trunc}
Let Assumption \ref{lin-terms}, Assumption
 \ref{assum-F} and Assumption \ref{assum-G} hold. Then, for each $n\ge 1$ the truncated equation \eqref{Truncated-Eq} admits a unique global solution $\bu^n\in M^2(X_T)$ for any $T\in (0,\infty)$.
\end{thm}
\begin{proof}
Let $n$ be a positive integer and  $\delta_n>0$ such that $\Lambda^n_{\delta_n, \bu_0}$ is a $\frac12$-contraction. To keep the notation simple we just write
$\delta:=\delta_n.$ For $k\in \mathbb{N}$ let $(t_k)_{k\in \mathbb{N}}$
  be a sequence of times defined by $t_k=k \delta$. By the $\frac12$-contraction property of $\Lambda^n_{\delta,u_0}$ we can find $\bu^{[n,1]}\in M^2(X_\delta)$ such that
  $\bu^{[n,1]}=\Lambda^n_{\delta,u_0}(\bu^{[n,1]})$. Since $\bu^{[n,1]}\in M^2(X_\delta)$ it follows from Lemma \ref{Weak-continuity} that $\bu^{[n,1]}$ is $\mathcal{F}_t$-measurable  and
  $\bu^{[n,1]}(t)\in L^2(\Omega, \mathbb{P};\ve)$ for any $t\in [0,\delta]$. Thus replacing $\bu_0$ with $\bu^{[n,1]}(\delta)$ where
  $$\bu^{[n,1]}(\delta):=\bu^{[n,1]}(\delta-)+\int_Z G(z,\bu_{n,\delta}(\delta-))\wie(dz,\{\delta\}),$$ and using the same argument as
   above we can find
  $\bu^{[n,2]}\in M^2(X_{t_1, t_2})$ such that $\bu^{[n,2]}=\Lambda^n_{\delta, \bu^{[n,1]}(\delta)}(\bu^{[n,2]})$.
  By induction we can construct a sequence $\bu^{[n,k]}\subset M^2(X_{t_{k-1},t_{k}})$ such that $\bu^{[n,k]}=\Lambda_{\delta, \bu^{[n,k-1]}}(\bu^{[n,k]})$.
  Now let  $\bu^n$ be the process defined by $\bu^n(t)=\bu^{[n,1]}(t)$, $t \in [0,\delta)$, and for $k=[\frac T \delta]+1$ and $0\le t<\delta$,
  let $\bu^{n}(t+k\delta)=\bu^{[n,k]}(t)$. By construction $\bu^n\in M^2(X_T)$ and $\bu^n= \Lambda^n_{T,u_0}(\bu^n)$, consequently $\bu^n$ is a global solution to
  the truncated equation \eqref{Truncated-Eq}.

  Now let $(\bv,\tau)$ be a another local solution of Eq. \eqref{Truncated-Eq}, we shall show that $\bu^n(t)=\bv(t)$, for all $t\in [0,\tau)$ almost surely. For this purpose let
  $t_1=\tau\wedge \delta$ and $t_k=\tau\wedge (k\delta)$ where $k$ and $\delta$ are as above; note that as $k\rightarrow \lfloor \frac T \delta\rfloor$ we have $t_k\uparrow \tau$ almost surely. With the same contraction principle used above we infer that
  $1_{[0,\tau \wedge\delta)}\bu^n(.)=1_{[0,\tau\wedge  \delta)} \bv(.)$ and $1_{[0,t_k)}\bu^n(.)=1_{[0,t_k)} \bv(.)$ almost surely. By letting $k\rightarrow \infty$ we infer that
  $\bu^n(t)=\bv(t)$, for all $t\in [0,\tau)$ almost surely.
\end{proof}

 \subsection{Existence and uniqueness of maximal/global solution to Eq. \eqref{Full-EQ}}
 In this subsection we will use what we have learnt from the solvability of the truncated equation \eqref{Truncated-Eq} to construct a unique maximal local and global solution to
 the original problem \eqref{Full-EQ}.

 We start with the existence and uniqueness of maximal local solution.
\begin{thm}\label{MAX-FULL-EQ}
 Let Assumption \ref{lin-terms}-\ref{assum-G} be satisfied, then there exists a unique pair $(\bu,\tau_\infty)$ which is a maximal local solution to \eqref{Full-EQ}.
\end{thm}

\begin{proof}
We have seen that for each $n\in \mathbb{N}$ Eq. \eqref{Truncated-Eq} has an unique global strong solution $\bu^n$.
Let us construct a sequence of stopping times
$\{\tau_n, n\in \mathbb{N}\}$
as follows
\begin{equation*}
 \tau_n=\inf\{t \ge 0, \lve \bu^n\rve_{X_t}\ge n\}\wedge T, n\in \mathbb{N}.
\end{equation*}
 Now let $k>n$ and $\tau_{n,k}=\inf\{t\ge 0, \lve \bu^k\rve_{X_t}\ge n\}\wedge T $. Since $\tau_{n.k}\le \tau_k$ a.s., $(\bu^k, \tau_{n,k})$ is a local solution to
Eq. \eqref{Truncated-Eq} and $(\bu^n, \tau_n)$ is also a local solution to Eq. \eqref{Truncated-Eq}. Hence by the uniqueness we proved in Theorem \ref{Global-Trunc} we infer that $\bu^n(t)=\bu^k(t)$
 a.s for all $t\in [0, \tau_n\wedge \tau_{n,k})$ which implies that
 \begin{equation}\label{Uniq-local}
\bu^n(t)=\bu^k(t) \text{ a.s. for }  t\in [0,\tau_n).
 \end{equation}
This also proves that $\tau_n<\tau_k$ a.s. for all $n<k$,
 and the sequence  $\{\tau_n, n\in \mathbb{N}\}$ has a limit
 $\tau_\infty:= \lim_{n\uparrow \infty}\tau_n$ a.s..

 Now let $\{\bu(t), 0\le t <\tau_\infty\}$ be the stochastic process defined by
 \begin{equation}
  \bu(t)=\bu^n(t), \,\, t\in [\tau_{n-1}, \tau_n), \,\, n\ge 1,
 \end{equation}
 where $\tau_0=0$. Since by definition $\theta_n(\lve \bu^n(s)\rve_{X_s})=1$ for any $s\in [0,t\wedge \tau_n)$,
 it follows that $B_n^T(\bu^n(s))=F(\bu^n(s))$ for any
$s\in [0,t\wedge \tau_n)$. By \eqref{Uniq-local} we have $\bu(t\wedge \tau_n)=\bu^n(t\wedge \tau_n)$, thus we can derive that $\mathbb{P}$-a.s.
\begin{align}
 \bu(t\wedge \tau_n)=& \bu_0-\int_0^{t\wedge \tau_n} \Big[A\bu^n(s)+B_n^T(\bu^n(s))\Big]ds+\int_0^{t\wedge \tau_n} \int_Z G(z,\bu^n(s))\wie(dz,s),\nonumber\\
 =& \bu_0-\int_0^{t\wedge \tau_n} \Big[A\bu(s)+F(\bu(s))\Big]ds+\int_0^{t\wedge \tau_n}\int_Z G(z,\bu(s))\wie(dz,s),\nonumber
\end{align}
for any $t\in [0,T]$. This proves that $(\bu, \tau_n)$ is a local solution to \eqref{Full-EQ}.
On $\{\tau_\infty(\omega) <T\}$ we have
\begin{align*}
 \lim_{t\uparrow \tau_\infty} \lve \bu \rve_{X_t}\ge& \lim_{n\uparrow \infty}\lve \bu \rve_{X_{\tau_n}},\\
 \ge& \lim_{n\uparrow \infty} \lve \bu^n \rve_{X_{\tau_n}}=\infty.
\end{align*}
Therefore $(\bu,\tau_\infty)$ is a maximal local solution to Eq. \eqref{Full-EQ}.

We will prove that this maximal solution is unique. For this let $(\bv, \sigma_\infty)$ be another maximal local solution and $\{\sigma_n, n\ge 0\}$
a sequence of stopping times converging to
$\sigma_\infty$ defined by
\begin{equation*}
 \sigma_n=\inf\{t\ge 0, \lve \bv\rve_{X_t}\ge n\}\wedge \sigma_\infty  \wedge T.
\end{equation*}
Arguing as above we can prove that
$\bu(t)=\bv(t)$ for all $t\in [0,\tau_n \wedge \sigma_n]$ a.s. which, upon letting $n\uparrow \infty$, implies that
\begin{equation*}
 \bu(t)=\bv(t) \text{ for all } t\in [0,\tau_\infty \wedge \sigma_\infty] \text{ a.s.}.
\end{equation*}
From this last identity we can conclude that $\tau_\infty=\sigma_\infty$ almost surely. Indeed if the last conclusion were not true then we either have
\begin{align}
 \lim_{t\uparrow \sigma_\infty}\lve 1_{\{\sigma_\infty>\tau_\infty \}} \bv \rve_{X_{\sigma_t}}
 =& \lim_{n\uparrow }\lve 1_{\{\sigma_\infty>\tau_\infty \}} \bv \rve_{X_{\sigma_n}}\nonumber\\
 =& \lim_{n\uparrow }\lve 1_{\{\sigma_\infty>\tau_\infty \}} \bu \rve_{X_{\tau_n}}=\infty,\label{Uniq-MAX-1}
\end{align}
or
\begin{align}
 \lim_{t\uparrow \tau_\infty}\lve 1_{\{\sigma_\infty<\tau_\infty \}} \bu \rve_{X_t}
 =& \lim_{n\uparrow }\lve 1_{\{\sigma_\infty<\tau_\infty \}} \bu \rve_{X_{\tau_n}}\nonumber\\
 =& \lim_{n\uparrow }\lve 1_{\{\sigma_\infty<\tau_\infty \}} \bv \rve_{X_{\sigma_n}}=\infty\label{Uniq-MAX-2}.
\end{align}
The identity \eqref{Uniq-MAX-1} (resp. Eq. \eqref{Uniq-MAX-2}) contradicts the fact that $\v$ (resp. $\bu$) does not explode before time $\sigma_\infty$
(resp. $\tau_\infty$).
Therefore one must have $\tau_\infty=\sigma_\infty$ almost surely, which yields the uniqueness of the maximal local solution $(\bu, \tau_\infty)$.
\end{proof}

\begin{prop}\label{THM-EXIST}
 In addition to the assumptions of Theorem \ref{GLO-FULL-EQ} we assume that $\me \lvert \bu_0\rvert^4< \infty$ and there exists $\tilde{C}_A>0$ such that
 $\langle A u, u\rangle \ge \tilde{C}_A \lVert u\rVert^2$ for any $u\in \ve$. We also suppose that $F$
 \begin{equation}\label{FNU0}
  \langle F(u), u\rangle=0,
 \end{equation}
 for all $u\in \ve$.
 Let $\bu$ be the stochastic process we constructed in Theorem \ref{MAX-FULL-EQ}.
 Let $(\tau_n)_{n\ge 1}$ be a sequence of stopping times defined by
 \begin{equation*}
  \tau_n=\inf\{t\ge 0: \lve \bu \rve^2_{X_t} \ge n^2\}\wedge T.
 \end{equation*}
Then for  $r=1,2$, for any $t\ge 0$ there exists constant $\tilde{C}>0$ such that the local solution $ (\bu, \tau_n)$ to \eqref{Full-EQ} satisfies
 \begin{equation}\label{WEAK-EST-1}
  \me \sup_{s\in [0,t\wedge \tau_n)}\lvert \bu(s)\lvert^{2r} +\me \int_0^{t\wedge \tau_n} \lvert \bu(s) \rvert^{2r-2} \lve \bu(s)\rve^2 ds\le \tilde{C},
 \end{equation}
 and
 \begin{equation}\label{WEAK-EST-2}
  \me \biggl[\int_0^{t\wedge \tau_n} \lve \bu(s)\rve^2 ds \biggr]^2 \le \tilde{C},
 \end{equation}
 for any $n\ge 1$.
\end{prop}
\begin{proof}
Let $\Psi(\bu):=\lvert \bu\rvert^2$ and
\begin{equation*}
\begin{split}
g(s,z,\bu):=\langle G(z,\bu(s)),G(z,\bu(s))\rangle,\\
 f(z,s,\bu):=g(s,z,\bu) +2 \langle G(z,\bu(s)), \bu(s-)\rangle.\\
\end{split}
\end{equation*}
 Note that $(\bu,\tau_\infty)$, where $\tau_\infty=\lim_{n\uparrow \infty}\tau_n $ a.s., is the unique maximal solution to \eqref{Full-EQ}. Throughout let $n$ be a fixed positive integer.
   To shorten notation we define $t_n=t\wedge \tau_n$ for any $t\in [0,T]$.

 Estimates \eqref{WEAK-EST-1} can be proved by using the It\^o' formula to $[\Psi(\bu(t_n))]^{r}$, $r=1,2$, with $t_n=t\wedge \tau_n$ for every $t\in [0,T]$. For $r=1$ the same calculations
 with $N=Id$ as in proof of Theorem \ref{GLO-FULL-EQ} yields
 \begin{equation}\label{EST-GLO}
\begin{split}
 \me \sup_{s\in[0, t_n)} \Psi(\bu(s))+2\me \int_0^{t_n} \langle A\bu(s),  \bu(s)\rangle ds\le \me \Psi(\bu_0)
 + \bar{\ell}_1 \me \int_0^{t_n} \lvert \bu(s)\rvert^2 ds\\
 + \bar{\ell}_1 T +\me \sup_{s\in [0,t_n)}\biggl\vert  \int_0^{s}\int_Z f(z,s,\bu) \wie(dz,ds)\biggr\vert,
 \end{split}
\end{equation}
for any $t\in [0,T]$ and $n\ge 1$.
Arguing as in the proofs of Eq. \eqref{ineq-ST-te} and \eqref{ineq-ST-te-2} we obtain the following inequality
\begin{equation*}
 \begin{split}
  \me \sup_{s\in [0,t_n]}\biggl\lvert \int_0^s \int_Z f(z, s ,\bu)\wie(dz,ds)\biggr\rvert\le
   \Big[\frac{C^2  \bar{\ell}_1}{4 \eps}+  \bar{\ell}_2 \Big] \me \int_0^{t_n} \lvert \bu(s) \rvert^2 ds \\
  +\eps \me\biggl[\sup_{s\in [0,t_n]}\lvert \bu(s)\rvert^2\biggr],
 \end{split}
\end{equation*}
which along with \eqref{EST-GLO} implies that
\begin{equation*}
\begin{split}
 \me \sup_{s\in[0, t_n)} \Psi(\bu(s))+2\me \int_0^{t_n} \langle A\bu(s),  \bu(s)\rangle ds\le
 \me \Psi(\bu_0)+\eps \me\biggl[\sup_{s\in [0,t_n]}\lvert \bu(s)\rvert^2\biggr] \\
 +
  \Big[\bar{\ell}_1 \Big[\frac{C^2   \bar{\ell}_1}{4 \eps}+1\Big]+ \bar{\ell}_2 \Big]\me \int_0^{t_n} \lvert \bu(s) \rvert^2 ds.
 \end{split}
\end{equation*}
Using Assumption \eqref{lin-terms}, choosing $\eps=1/2$ and invoking the Gronwall lemma yield
\begin{equation}\label{EST-GLO-1}
 \me \sup_{s\in [0,t_n)}\lvert \bu(s)\rvert^2 + \me \int_0^{t_n} \lve \bu(s)\rve^2 ds\le \frac{\me \Psi(\bu_0)}{\min(\frac{1}{2}, 2\tilde{C}_A)}[e^{\ell T} +1],
\end{equation}
where $$\ell = \frac{}{\min(\frac{1}{2}, 2\tilde{C}_A)} \Big[\bar{\ell}_1 \Big[\frac{C^2   \bar{\ell}_1}{2 1}+1\Big]+ \bar{\ell}_2 \Big].$$
This completes the proof of the theorem for $r=1$.

Now, for $t\ge 0$ let $y(t):=\langle \bu(t),  \bu(t) \rangle$ and $\Psi^\prime(\bu(t))[h]=2 \langle  \bu(t), h \rangle$ for any $h\in \h$.

 First we should notice that by It\^o's formula and the assumption about $F$ in Proposition \ref{THM-EXIST} we have
 \begin{equation}\label{EST-GLO-2}
 \begin{split}
  y(t_n) +\int_0^{t_n} \Psi^\prime(\bu(s))[A\bu(s)] ds= y(0) +\int_0^{t_n} \int_Z g(s,z,\bu) \nu(dz) ds
  +\int_0^{t_n }\int_Z f(s,z,\bu) \wie(dz,ds).
  \end{split}
 \end{equation}
 By applying It\^o's formula to $[y(t)]^2=:z(t)$ we obtain
 \begin{equation}\label{EST-GLO-3}
 \begin{split}
  \me \sup_{r\in [0,t_n)}\biggl[z(r) +2 \int_0^{r} y(s) \Psi^\prime(\bu(s))[A\bu(s)] ds-2 \int_0^{r} \int_Z y(s) g(s,z,\bu) \nu(dz) ds\biggr]\\
  = \me \sup_{r\in [0,t_n)}\biggl[z(0)  +\int_0^{r}\int_Z [f(s,z,\bu)]^2  \nu(dz)ds\biggr]\\
  + \me \sup_{r\in [0,t_n)} \biggl[\int_0^{r}\int_Z \Big( [f(s,z,\bu)]^2 +2 y(s-)  f(s,z,\bu)\Big)\wie(dz,ds)\biggr].
  \end{split}
 \end{equation}
 By performing elementary calculation and using part \eqref{part-ii} of  Assumption \ref{assum-G} one can show that
 \begin{align}
  % \begin{split}
   \me \int_0^{t_n }\int_Z [f(s,z,\bu)]^2  \nu(dz)ds \le & 2 \me \int_0^{t_n} \int_Z [\langle  G (z,\bu(s)), G(z,\bu(s))\rangle]^2 \nu(dz) ds\nonumber\\
   & \quad \quad + 2\me
\int_0^{t_n} \int_Z [\langle  \bu(s-), G(z,\bu(s))\rangle]^2 \nu(dz) ds,\nonumber \\
\le & 2 [\bar{\ell}_1+\bar{\ell}_2^2]\Big(T+ \me \int_0^{t_n} \lvert \bu(s) \rvert^4 ds\Big).\label{EST-GLO-4}
  %\end{split}
 \end{align}
 Similarly,
 \begin{equation}\label{EST-GLO-4-b}
  2\me \int_0^{t_n} \int_Z y(s) g(s,z,\bu) \nu(dz) ds\le  \bar{\ell}_1 \Big(T +\me \int_0^{t_n} \lvert \bu(s)\rvert^4 ds\Big)
 \end{equation}
Since $[\psi(s,z,\bu)]^2 >0 $, by using \cite[Theorem 3.10, Eq. (3.10)]{Rud+Zig} we derive that
 \begin{equation*}
 \begin{split}
  \me \sup_{r\in[0,t_n) } \biggl \lvert \int_0^{r}\int_Z  [f(s,z,\bu)]^2 \wie(dz,ds)\biggr\lvert \le \me \int_0^{t_n}\int_Z [f(s,z,\bu)]^2 \nu(dz)ds
 \end{split}
 \end{equation*}
and by arguing as above we infer that
\begin{align}
%\begin{split}
  \me \int_0^{t_n}\int_Z [f(s,z,\bu)]^2 \wie(dz,ds)\le
 2 ^2 [\bar{\ell}_1+\bar{\ell}_2^2]\Big(T+ \me \int_0^{t_n} \lvert \bu(s) \rvert^4 ds\Big)\label{EST-GLO-5}.
%  \del{+ 4 ^2 [\bar{\ell}_1+1]\Big(T+ \me \int_0^{t_n} \lvert \bu(s) \rvert^4 ds\Big),\nonumber \\}
%  \del{\me \int_0^{t_n}\int_Z \Big( [\psi(s,z,\bu)]^2 +2 y(s)  \psi(s,z,\bu)\Big)\wie(dz,ds)
%  \le 2 ^2 [1+\bar{\ell}_1+\bar{\ell}_2^2]\Big(T+ \me \int_0^{t_n} \lvert \bu(s) \rvert^4 ds\Big).}
%\end{split}
 \end{align}
 By using the BDG inequality and Cauchy inequality with epsilon we obtain
  \begin{align}
 \me \sup_{r\in[0,t_n) } \biggl \lvert \int_0^{r}\int_Z 2 y(s-)  f(s,z,\bu)\wie(dz,ds)\Biggr\lvert\le & 4K \me \biggl[\int_0^{t_n} \int_Z   [y(s)]^2
  \lvert f(s,z,\bu)\rvert^2 \Big)\nu(dz)ds \biggr]^\frac12,\nonumber \\
  \le & \frac{16 K^2 }{4\eps} \me \int_0^{t_n} \int_Z   \lvert f(s,z,\bu)\rvert^2 \Big)\nu(dz)ds\nonumber\\
 & \quad \quad \quad +\eps \me \sup_{s\in [0,t_n)}\lvert \bu(s)\rvert^4.
  \end{align}
And arguing as in \eqref{EST-GLO-5} we derive that
 \begin{equation}\label{EST-GLO-5-b}
  \begin{split}
   \me \sup_{r\in[0,t_n) } \biggl \lvert \int_0^{r}\int_Z 2 y(s-)  f(s,z,\bu)\wie(dz,ds)\Biggr\lvert-\eps \me \sup_{s\in [0,t_n)}\lvert \bu(s)\rvert^4
   \\ \le \frac{32K^2 (1+\bar{\ell}_1+\bar{\ell}_2^2)}{4\eps} \Big(T+\me \int_0^{t_n}  \lvert\bu(s)\rvert^4 ds\Big).
  \end{split}
 \end{equation}

 Plugging \eqref{EST-GLO-4}, \eqref{EST-GLO-4-b}, \eqref{EST-GLO-5} and \eqref{EST-GLO-5-b} in \eqref{EST-GLO-3}, using Assumption \ref{lin-terms}
 and choosing $\eps=1/2$ yield the existence of positive constants $\tilde{L}$, $\bar{\ell}$ such that
 \begin{equation*}
  \begin{split}
   \me \sup_{s\in[0,t_n)}\lvert \bu(s)\rvert^4 +\int_0^{t_n} \lvert \bu(s)\rvert^2 \lve \bu(s)\rve^2 ds\le
   \tilde{L} T+ \tilde{L} \me \int_0^{t_n} \lvert \bu(s) \rvert^4 ds
   +\bar{\ell} \me[\Psi(\bu_0)]^2.
  \end{split}
 \end{equation*}
Thanks to the Gronwall lemma we infer that
\begin{equation}
 \me \sup_{s\in[0,t_n)}\lvert \bu(s)\rvert^4 +\int_0^{t_n} \lvert \bu(s)\rvert^2\lve \bu(s)\rve^2 ds\le \big(\tilde{L}T+ \bar{\ell} \me [\Psi(\bu_0]^2\big)
 \big[e^{\tilde{L} T} +1\big].
\end{equation}
The above inequality completes the proof of \eqref{WEAK-EST-1} for $r=2$, and hence the first part of our theorem.

To prove the second part we will use \eqref{EST-GLO-1}. In fact, from \eqref{EST-GLO-1} we derive that
\begin{equation}\label{EST-GLO-6-a}
\begin{split}
 \me \biggl[\int_0^{t_n} \Psi^\prime (\bu(s))[A\bu(s)] ds\biggr]\le C \me [y(0)]^2+C \me \biggl[\int_0^{t_n} \int_Z g(s,z, \bu) \nu(dz) ds\biggr ]^2 \\
 +C \me \biggl[\int_0^{t_n} \int_Z f(s,z,\bu) \tilde{\eta}(dz,ds) \biggr]^2.
 \end{split}
\end{equation}
Note that the stochastic integral in the last term of the RHS of the above estimate is real-valued, so from It\^o's isometry we infer that
\begin{align*}
 \me \biggl[\int_0^{t_n} \int_Z f(s,z,\bu) \tilde{\eta}(dz,ds) \biggr]^2=\me \int_0^{t_n} \int_Z [f(s,z,\bu)]^2 \nu(dz)ds,\nonumber
\end{align*}
from which altogether with \eqref{EST-GLO-5} and \eqref{WEAK-EST-1} we derive that for any $t\ge 0$ there exists a constant $\tilde{C}>0$ such that
\begin{equation}\label{EST-GLO_6}
 \me \biggl[\int_0^{t_n} \int_Z f(s,z,\bu) \tilde{\eta}(dz,ds) \biggr]^2\le \tilde{C}
\end{equation}
for any $n\ge 1$. By imitating the proof of \eqref{EST-GLO-4} we infer that there exists $C>0$ such that
\begin{equation}
 \me \biggl[\int_0^{t_n} \int_Z g(s,z, \bu) \nu(dz) ds\biggr ]^2 \le C \Big(t+ \me \int_0^{t_n} \lvert \bu(s) \rvert^4 ds\Big),
\end{equation}
from which and \eqref{WEAK-EST-1} we deduce that for any $t\ge 0$ there exists $\tilde{C}>0$ such that
\begin{equation}\label{EST-GLO-7}
  \me \biggl[\int_0^{t_n} \int_Z g(s,z, \bu) \nu(dz) ds\biggr ]^2 \le \tilde{C},
\end{equation}
for any $n\ge 1$.
Taking \eqref{EST-GLO_6} and \eqref{EST-GLO-7} into \eqref{EST-GLO-6-a} implies that
\begin{equation}
\begin{split}
 \me \biggl[\int_0^{t_n} \Psi^\prime (\bu(s))[A\bu(s)] ds\biggr]\le C \me [y(0)]^2+2 \tilde{C}.
 \end{split}
\end{equation}
Thanks to this last estimate and the fact that $\langle A \bu, \bu \rangle \ge \tilde{C}_A \lVert \bu\rVert^2$ we easily derive that for any $t\ge 0$ there exists
$\tilde{C}>0$ such that for any $n\ge 1$
\begin{equation*}
 \me \biggl[\int_0^{t_n} \lVert \bu(s) \rVert^2 ds\biggr]\le  \tilde{C}.
\end{equation*}
This completes the proof of \eqref{WEAK-EST-2}, and hence the whole Proposition.
\end{proof}

Now we turn our attention to the existence and uniqueness of global solution.
\begin{thm}\label{GLO-FULL-EQ}
Assume that $F$ satisfies the assumptions of Proposition \ref{THM-EXIST} with $p=1$ and $\alpha\in [0, \frac12]$. Moreover, we suppose that there exists $\tilde{c}>0$ such that
\begin{equation}\label{Cond-GLO}
\begin{split}
 \lvert F(u)-F(v)\rvert \le \tilde{c}\biggl[\lvert u\rvert^{1-\alpha} \lve u\rve^\alpha \lve u-v\rve^{1-\alpha} \lve u-v\rve_\ast^\alpha
 + \lvert u-v\rvert^{1-\alpha} \lve u-v\rve^\alpha \lve v\rve^{1-\alpha} \lve v\rve_\ast^\alpha\biggr]
\end{split}
\end{equation}
for any $u,v\in \be$.\del{
 \begin{equation}\label{FNU0}
  \langle F(u), N u\rangle=0,
 \end{equation}
 for all $u\in H$.} Then Problem \eqref{Full-EQ} has a unique global solution.
\end{thm}
\begin{proof}
Let $\bu$ be the stochastic process we constructed in Theorem
\ref{MAX-FULL-EQ} and $$\Lve N\Lve:=\max\left(\noe, \nov\right).$$
 Let $(\tau_n)_{n\ge 1}$ be a sequence of stopping times defined by
 \begin{equation*}
  \tau_n=\inf\{t\in [0,T]: \lve \bu(t) \lve^2 +\int_0^t \lve \bu(s) \rve_\ast^2 ds \ge n^2\}.
 \end{equation*}
 Note that $(\bu,\tau_\infty)$, where $\tau_\infty=\lim_{n\uparrow \infty}\tau_n $ a.s., is the unique maximal solution to \eqref{Full-EQ}.
 To deal with the structure of the nonlinearity $F$ (see Eq. \eqref{Cond-GLO}) we introduce another sequence of stopping times $(\sigma_m)_{m\ge 1}$ defined by
 \begin{equation*}
  \sigma_m=\inf \biggl\{t\in [0,T]:\int_0^t \lvert \bu(s) \rvert^2 \lVert \bu(s)\rVert^{\frac{2\alpha}{1-\alpha}} ds\ge m \biggr\}, \text{ for any } m\ge 1.
 \end{equation*}

 To shorten notation we define $t_{m,n}=t\wedge (\sigma_m \wedge \tau_n)$ for any $t\in [0,T]$, $n\ge 1$ and $m\ge 1$.
 Let
\begin{equation*}
\begin{split}
 f(z,s,\bu):=\langle N G(z,\bu(s)),G(z,\bu(s))\rangle\\
 +2 \langle G(z,\bu(s)), N\bu(s-)\rangle,
\end{split}
\end{equation*}
and
\begin{equation*}
 g(z,s,\bu):= \langle N G(z,\bu(s)), G(z,\bu(s))\rangle.
\end{equation*}
Applying It\^o's formula to $\Psi(u)=\langle u, Nu\rangle$ we obtain
\begin{equation*}
\begin{split}
 \Psi(\bu(t_{m,n}))=\Psi(\bu_0)-2\int_0^{t_{m,n}} \Big[\langle A\bu(s)+F(\bu(s)), N \bu(s) \rangle \Big] ds+
 \int_0^{t_{m,n}}\int_Z g(z,s, \bu) \nu(dz) ds \\
 +\int_0^{t_{m,n}}\int_Z f(z,s,\bu) \wie(dz,ds),
 \end{split}
\end{equation*}
for any $t\in [0,T]$.
For any $\delta>0$ and $p,q\ge 1$ with $p^{-1}+q^{-1}=1$ let $C(\delta,p,q)$ be the constant from the Young inequality  $$ab\le C(\delta,p,q)a^p+\delta
b^q.$$ From Eq. \eqref{Cond-GLO} and the above Young inequality with $p=\frac{2}{1+\alpha}$,
$q=\frac{2}{1-\alpha}$, and $\delta=C_A$ we obtain
\begin{equation*}%\label{Cond-GLO-2}
 \lvert 2\langle F(\bu(s)), N\bu(s)\rangle \rvert \le C(C_A,p,q) [2 \tilde{c} \Lve N\Lve ]^q  \lvert \bu(s) \rvert^2 \lVert \bu(s)\rVert^{\frac{2\alpha}{1-\alpha}  }
 +C_A
 \lve \bu(s)\rve_\ast^2.
\end{equation*}
By making use of the definition of $\sigma_m$ we get that
\begin{equation}\label{Cond-GLO-2}
 2 \lvert  \int_0^{t_{m,n}} \langle F(\bu(s)), N\bu(s)\rangle ds \rvert \le
 C(C_A,p,q) [2 \tilde{c} \Lve N\Lve ]^q  m T
 +C_A\int_0^{t_{m,n}}
 \lve \bu(s)\rve_\ast^2.
\end{equation}
From the assumption on $G$ we derive that
\begin{equation}\label{Cond-GLO-3}
 \int_0^{t_{m,n}}\int_Z g(z,s, \bu) \nu(dz) ds\le \Lve N\Lve \tilde{\ell}_1 T+\Lve N\Lve \tilde{\ell}_1 \int_0^{t_{m,n}} \lve \bu(s)\rve^2 ds.
\end{equation}

% Now, let us set $\rho(s)=e^{-\int_0^{s} \phi(r) dr}$, where the function $\phi$ is defined by
% $$\phi(s)=C(1 ,p,q)[2 \tilde{c} \Lve N\Lve ]^q \lvert \bu(s)\rvert^{2} \lve \bu(s)\rve^{\frac{2\alpha}{1-\alpha}  }.$$
% Applying It\^o's formula to the stochastic process $\rho(t) \Psi(\bu(t))$ and using Eqs. \eqref{Cond-GLO-2}, \eqref{Cond-GLO-3} yield
% \begin{equation*}
%  \begin{split}
%   \rho(t_{m,n})\Psi(\bu(t_{m,n}))+2\int_0^{t_{m,n}} \rho(s) \Psi^\prime(\bu(s))[A\bu(s)] ds\le \Psi(\bu_0)+C_A\int_0^{t_{m,n}} \rho(s) \lve \bu(s)\rve_\ast^2 ds
%  \\ + \Lve N\Lve \tilde{\ell}_1 T +\Lve N\Lve \tilde{\ell}_1 \int_0^{t_{m,n}} \rho(s) \lve \bu(s)\rve^2 ds\\
%  + \int_0^{t_{m,n}} \int_Z \rho(s) f(s,z, \bu) \wie(dz,ds).
%  \end{split}
% \end{equation*}
By taking the mathematical expectation to both sides of this estimate and by using Assumption \ref{lin-terms} altogther with Eqs.
\eqref{Cond-GLO-2}, \eqref{Cond-GLO-3}
 we infer that
\begin{equation*}
\begin{split}
 \me [\lve \bu(t_{m,n})\rve^2 ] +\me \int_0^{t_{m,n}} \lve \bu(s)\rve_\ast^2 ds\le \tilde{L}^{-1} \Lve N\Lve \tilde{\ell}_1
 \me \int_0^{t_{m,n}} \lve \bu(s)\rve^2 ds
 \\
 +\tilde{L}^{-1} [\me \Psi(\bu_0)+ \Lve N\Lve \tilde{\ell}_1 T+C_{mA}T],
 \end{split}
\end{equation*}
where $\tilde{L}=\min(C_N, C_A)$ and $C_{mA}:=C(C_A,p,q) [2 \tilde{c} \Lve N\Lve ]^q  m $.
From the Gronwall's lemma we infer that
\begin{equation}\label{ST15}
 \me [ \lve\bu(t_{m,n})\rve^2 ] +\me \int_0^{t_{m,n}}  \lve \bu(s)\rve_\ast^2 ds\le \tilde{L}^{-1} [\me \Psi(\bu_0)+ \Lve N\Lve \tilde{\ell}_1 T+C_{mA}T]
 e^{\tilde{L}^{-1} \Lve N\Lve \tilde{\ell}_1 t_{m,n}}[1+\Lve N\Lve \tilde{\ell}_1 T].
\end{equation}
Next, note that
\begin{align*}
 \mathbb{P}(\tau_n<t)=& \mathbb{P}(\{\tau_n<t\}\cap (\Omega_m\cup \Omega_m^c)),\\
 =& \mathbb{P}(\{\tau<t_n \}\cap \Omega_m)+ \mathbb{P}(\{\tau_n<t\}\cap \Omega_m^c),
\end{align*}
where $\Omega_m=\{\sigma_m\ge T\}$, $m\ge 1$. Now, by arguing as
in \cite[pp. 123]{ZB et al 2005} we have
\begin{equation*}
\begin{split}
\mathbb{P}\left(\tau_n<t\right)&\le \frac 1 {n^2} \mathbb{E}\biggl(1_{\{\tau_n<t\}\cap \Omega_m}
\biggl[\lve \bu({t_{m,n}}) \lve^2 +\int_0^{t_{m,n}} \lve \bu(s) \rve_\ast^2 ds
\biggr]\biggr)
+\mathbb{P}[\Omega_m^c],\\
&\le \frac 1 {n^2} \mathbb{E}\biggl[\lve \bu({t_{m,n}}) \lve^2 +\int_0^{t_{m,n}} \lve \bu(s) \rve_\ast^2 ds\biggr] +
\frac1m \me \int_0^{t_{m,n}} \lvert \bu(s) \rvert^2 \lve\bu(s)\rve^\frac{2\alpha}{1-\alpha} ds.%\\
% \le & \frac 1 {n^2} \me \Big(\rho(t_{m,n})\lve\bu(t_{m,n})\rve^2
% \Big)+\frac{ C(C_A,p,q)}{n^2\log{n}}\me \int_0^{t_{m,n}} \lvert \bu(s) \rvert^2 \lve\bu(s)\rve^\frac{2\alpha}{1-\alpha} ds.
\end{split}
\end{equation*}
Thanks to Eq.
\eqref{ST15}
\begin{equation*}
 \begin{split}
  \mathbb{P}\left(\tau_n<t\right)&\le  \frac 1 {n^2} \tilde{L}^{-1} [\me \Psi(\bu_0)+ \Lve N\Lve \tilde{\ell}_1 T+C_{mA}T]
 e^{\tilde{L}^{-1} \Lve N\Lve \tilde{\ell}_1 T} +
\frac1m \me \int_0^{t_{m,n}} \lvert \bu(s) \rvert^2 \lve\bu(s)\rve^\frac{2\alpha}{1-\alpha} ds,
 \end{split}
\end{equation*}
from which we derive that
\begin{equation*}
 \begin{split}
\lim_{n\toup \infty }\mathbb{P}\left(\tau_n<t\right)\le\frac1m \biggl\{\me \big[\sup_{s\in [0,t_{m,n}]} \lvert \bu(s) \rvert^4 \big]
 +\Big(\me \Big[\int_0^{t_{m,n}}  \lVert \bu(s)\rVert^2 ds\Big]^2\Big)^\frac{2\alpha}{1-\alpha}\biggr\}.
 \end{split}
\end{equation*}
Since $\alpha\in [0,\frac12]$ it follows from Proposition \ref{THM-EXIST} (see \eqref{WEAK-EST-1}-\eqref{WEAK-EST-2}) that the solution $\bu$
satisfies
\begin{equation*}
\me [\sup_{s\in [0,t_{m,n}]} \lve \bu(s) \rve^4 ] +\biggl(\me \biggl[\int_0^{t_{m,n}}  \lVert \bu(s)\rVert^2 ds\biggr]^2\biggr)^\frac{2\alpha}{1-\alpha} \le \tilde{C}.
\end{equation*}
% which implies that
% $$\E \int_0^{t_k} \lve \v\rve^2 \lve \rA^{\frac 12}
% \v \rve^2 ds< \infty.$$
Hence, combining this latter equation with the former one yields that
$$\lim_{n\rightarrow \infty}\mathbb{P}\left(\tau_n<t\right)=0,$$
from which we derive that $\mathbb{P}\Big(\tau_\infty<T\Big)=0 $ for any $T>0$. This implies that $\bu$ is a global solution.
\end{proof}

\begin{Rem}\label{REM-GLO}
All of our results in this section remain true if we replace $F(u)$ by $B(u)+R(u)$ with $R\in \mathcal{L}(H,H)$ and $B$ satisfying the assumptions of Theorem
\ref{MAX-FULL-EQ} and Theorem \ref{GLO-FULL-EQ}.
\end{Rem}

\section{Examples}\label{EXAM}
The examples, notations and references used in this section are
taken from \cite{Chueshov}.
\subsection{Notations}\label{notation}
 Let $n\in \{2,3\}$ and assume that $\MO \subset \mathbb{R}^n$ is a Poincar\'e's domain (its definition is given below) with
 boundary $\partial \MO$ of class $\mathcal{C}^\infty$.
 For any $p\in [1,\infty)$ and $k\in \mathbb{N}$,  $\el^p(\MO)$ and
$\mathbb{W}^{k,p}(\MO)$ are the well-known Lebesgue and Sobolev
spaces, respectively, of $\mathbb{R}^n$-valued functions. The
corresponding spaces of scalar functions we will denote by
standard letter, e.g. ${W}^{k,p}(\MO)$.

A domain $\CO\subset \RR^d$ is called a Poincar\'e's domains if following Poincar\'e's inequality holds
\begin{equation}\label{Poincare}
 \lvert \bu \rvert \le c \lvert \nabla \bu \rvert, \text{ for all } \bu \in H^1(\CO).
\end{equation}

 For $p=2$ we denote $\mathbb{W}^{k,2}(\MO)=\bh^k$ and its norm are denoted by $\lve \bu
 \rve_k$. By $\bh^1_0$ we mean the space of functions in $\bh^1$
 that vanish on the boundary on $\MO$; $\bh^1_0$ is a Hilbert space when endowed with the scalar product induced by that of $\bh^1$.
The usual scalar product on $\el^2$ is denoted by $\langle
u,v\rangle$ for $u,v\in \el^2$. Its associated norm is $\lvert
u\rvert$, $u\in \el^2$. We also introduce the following spaces
 \begin{align*}
\mathcal{V}_1& =\left\{ \bu\in [\mathcal{C}_{c}^{\infty }(\MO,\mathbb{R}^n)]\,\,\text{such that}%
\,\,\nabla \cdot \bu=0\right\} \\
\mathbf{V}_1& =\,\,\text{closure of $\mathcal{V}$ in }\,\,\mathbb{H}_0^{1}(\MO) \\
\mathbf{H}_1& =\,\,\text{closure of $\mathcal{V}$ in
}\,\,\mathbb{L}^{2}(\MO).
\end{align*}
We also consider the Hilbert spaces  $\h_{2}=\h_{1}$ and
$\bve_2= \bh^1 \cap \h_{2}$.

Let $(e_1, e_2)$ be the standard basis in
 $\RR^2$  and $x=(x^1,x^2)$ an element of $\RR^2$.
When $\CO =(0, l) \times (0, 1)$ is a rectangular   domain  in the
vertical  plane we consider the following spaces
\begin{align*}
\h_{3} =  & \left\{u\in \mathbb{L}^2,\; \text{div} u=0, \;
u^2 |_{x^2=0}=u^2 |_{x^2=1}=0,\; u^1 |_{x^1=0}=u^1 |_{x^1=l}
  \right\}    \\
\end{align*}
and $\h_{4}=    L^2(\CO)$.
We also denote
 \begin{align*}
\bve_3 =  & \left\{u\in \h_3\cap\bh^1,\;  u |_{x^2=0}=u |_{x^2=1}=0,\;
 u \; \mbox{is $l$-periodic in}\; x^1  \right\},    \\
\bve_4=  & \left\{\theta \in H^1(\CO),\;  \; \theta|_{x^2=0}=\theta|_{x^2=1}=0,\;
\theta \; \mbox{is $l$-periodic in}\; x^1
\right\},\\
\h_5= &\h_3,\\
\bve_5= &\h_5 \cap \bh^1.
\end{align*}
Let $\Pi_i: \el^2 \rightarrow \h_i$ be the projection
from $\el^2$ onto $\h_i$, $i=1,2,3,4,5$. We denote by $\rA_i$ the
Stokes operator defined by
\begin{equation}\label{STOKES}
\begin{cases}
 D(\rA_i)=\{ u\in \h_i, \; \Delta u \in \h_i\},\\
 \rA_i u =-\Pi_i\Delta u, \; u\in D(\rA_i),
\end{cases}
\end{equation}
$i=1, \ldots, 5$. In all cases the $\rA_i$-s are self-adjoint, positive linear operators on $\h_i$.
Finally we set $\bbe_i= D(\rA_i)$, $i \in \{1,2,3,4, 5\}$. Note that $\bbe_i\subset \bh^2 \cap \bve_i$, $i=1,2,3,5$ and $\bbe_4\subset H^2 \cap \bve_4.$

We endow the spaces $\h_i$, $i \in \{1,2,3,4, 5\}$, with the scalar product and norm of $\el^2$. We equip the space $\bve_i$, $i \in \{1,2,3,4, 5\}$,
with the scalar product
$\langle \rA_i^\frac 12 \bu, \rA_i^\frac 12  \bv\rangle$ which is equivalent to the
$\bh^1(\MO)$-scalar product on $\bve_i$. The spaces $\bbe_i$, $i \in \{1,2,3,4, 5\}$ are equipped with the norm $\lvert \rA_i \bu \rvert$ which is equivalent to the $\bh^2$-norm on
$\bbe_i$.
\begin{Rem}
 In the case of an general unbounded domain we equip the space $\bve_i$, $i \in \{1,2,3,4, 5\}$,
with the scalar product
$\langle (\Id+\rA_i)^\frac 12 \bu, (\Id+\rA_i)^\frac 12  \bv\rangle$. The spaces $\bbe_i$, $i \in \{1,2,3,4, 5\}$ are equipped with the norm $\lvert (\Id+\rA_i) \bu \rvert$ which is equivalent to the $\bh^2$-norm on
$\bbe_i$.
\end{Rem}

% Let $\kappa_i$, $i=1,2,3$, be non-negative numbers and let $a(\cdot, \cdot)$ be the form defined by
% \[
% a(z_1,z_2)=-\kappa_1 \sum\limits_{j=1}^{n}
%    \int_{D} \Delta u_1^j \cdot u_2^j \, dx-
% \kappa_2
%    \int_{D} \Delta \th_1 \cdot \th_2 \, dx
% -\kappa_3 \sum\limits_{j=1}^{n}
%    \int_{D} \Delta b_1^j \cdot  b_2^j \, dx\quad
%    \]
% for $z_1=(u_1,\th_1,b_1)\in \bh^2 \times H^2 \times \bh^2$ and $z_1=(u_2,\th_2,b_2)\in \el^2 \times L^2 \times \el^2$.

Next we define two trilinear forms $b_1(\cdot, \cdot, \cdot)$ and $b_2(\cdot, \cdot, \cdot)$ by setting
\begin{align}
 b_1(\bu, \bv, \bw)=\sum_{i,j=1}^n \int_O \bu^i(x) \frac{\partial}{\partial x_i}\bv^j(x) \bw^j(x) dx, \text{ for any }
 (\bu, \bv, \bw) \in \el^4\times \mathbb{W}^{1,4} \times \el^2,\label{TRI-B1}\\
 b_2(\bu, \th_2, \th_3)=\sum_{i=1}^n \int_O \bu^i(x) \frac{\partial}{\partial x_i}\th_2(x) \th_3(x) dx, \text{ for any }
 (\bu, \th_2, \th_3) \in \el^4\times {W}^{1,4} \times L^2.\label{TRI-B2}
\end{align}
% $(e_1, e_2)$ the standard basis in
%  $\RR^2$  and $x=(x^1,x^2)$ an element of $\RR^2$.
%  Denote by  $p(x,t)$ the pressure field,   $f, g$  external forces,
% $ u=(u^1(x,t),u^2(x,t))$ the velocity field and  $\th=\th(x,t)$
% the temperature field  satisfying the following system
% \begin{eqnarray}
% \partial_t u + u \nabla u-\nu \D u + \nabla p &=&  \th e_2  + f,
% \quad \di u   = 0, \label{1.1}\\
% \partial_t \th +u \nabla \th -u^2 -\k\D \th &=&
% g,\label{eqn3}
% \end{eqnarray}
% with boundary conditions
% \begin{eqnarray*}
% u =0\;\;  \& \;\; \th=0 \;\; \mbox{on}\;\;x^2=0\; \mbox{and}\;x^2=1,   \\
% u, p, \th, u_{x^1}, \th_{x^1} \; \mbox{are periodic in}\; x^1 \;
% \mbox{with period}\; l.\footnotemark
% \end{eqnarray*}
% \footnotetext{Here and below this means that $\phi |_{x^1=0}=\phi |_{x^1=l}$
% for the corresponding function.}
%
% Here above  $\nu$ is the kinematic viscosity,
% $\kappa$ is the thermal  diffusion coefficient.
% Let
Recall that for $\alpha=\frac{n}{4}$, the following estimate, valid for all $\bu\in \bh^1$ (or $\bu \in H^1$), is a special case of Gagliardo-Nirenberg's inequalities:
\begin{align}
\lve \bu\rve_{\el^4}\le \lvert \bu\rvert^{1-\alpha} \lvert \nabla \bu
\rvert^\alpha. \label{GAG-l4}
% \\
% \lve \bu \rve_{\el^\infty}\le \lve \bu \rve_{\el^4}^{1-a}\lve
% \nabla \bu \rve_{\el^4}^a.\label{GAG-LInf}
\end{align}
The inequality \eqref{GAG-l4} can be written in the spirit of
the continuous embedding
\begin{equation}\label{SOB-EM}
\bh^1\subset \el^4.
\end{equation}
Using Cauchy-Schwarz inequality, \eqref{GAG-l4} and \eqref{SOB-EM} in \eqref{TRI-B1}-\eqref{TRI-B2} we derive that
for any $(\bu, \bv, \bw) \in \bh^1 \times \bh^2 \times \el^2$
\begin{align}
 \lvert b_1(\bu, \bv, \bw)\rvert \le c \lve \bu \rve_{\bh^1}\,\, \lvert \nabla \bv \rvert^{1-\alpha} \,\, \lvert D^2 \bv \rvert^\alpha \,\, \lvert \bw \rvert
 \text{ for any $(\bu, \bv, \bw) \in \bh^1 \times \bh^2 \times \el^2$,}\label{BIL-B1} \\
 \lvert b_2(\bu, \th_2, \th_3) \rvert \le c \lve \bu \lve_{\bh^1}\,\, \lvert \nabla \th_2 \rvert^{1-\alpha}\,\,
 \lvert D^2 \th_2 \rvert^\alpha \,\, \lvert \th_3 \rvert
 \text{ for any $(\bu, \th_2, \th_3) \in \bh^1 \times H^2 \times L^2$.} \label{BIL-B2}
\end{align}
From Eq. \eqref{BIL-B1} (resp., Eq. \eqref{BIL-B2}) we infer that there exists a bilinear map $B_1(\cdot, \cdot)$ (resp.,  $B_2(\cdot, \cdot)$)
defined on $\bve_i \times \bbe_i$ and taking values in $\h_i$, for appropriate values of $i$. Moreover, there exist $c>0$ such that
\begin{align}
 \lvert B_1(\bu, \bv)\lvert \le c \lve \bu \rve_{\bh^1} \lve\,\, \lve \bv \lve_{\bh^1}^{1-\alpha} \lve \bv \lve_{\bh^2}^\alpha, \text{ for any }
 (\bu, \bv) \in \bve_i\times \bbe_i, \label{BIL-B1-A}\\
 \lvert B_2(\bu, \th_2)\lvert \le c \lve \bu \rve_{\bh^1} \lve\,\, \lve \th_2 \lve_{\bh^1}^{1-\alpha} \lve \th_2 \lve_{\bh^2}^\alpha, \text{ for any }
 (\bu, \bv) \in \bve_i\times \bbe_i, \label{BIL-B1-B}
\end{align}
for appropriate values of $i$. Note that using Cauchy-Schwarz inequality, \eqref{GAG-l4} and \eqref{SOB-EM} in \eqref{TRI-B1}-\eqref{TRI-B2} we also derive that
\begin{align}
 \lvert B_1(\bu, \bv)\lvert \le c \lvert \bu \rvert^{1-\alpha} \lve \bu \rve^\alpha_{\bh^1} \lve\,\, \lve \bv \lve_{\bh^1}^{1-\alpha} \lve \bv \lve_{\bh^2}^\alpha, \text{ for any }
 (\bu, \bv) \in \bve_i\times \bbe_i, \label{BIL-B1-C}\\
 \lvert B_2(\bu, \th_2)\lvert \le c \lvert \bu \rvert^{1-\alpha} \lve \bu \rve^\alpha_{\bh^1}\,\, \lve \th_2 \lve_{\bh^1}^{1-\alpha} \lve \th_2 \lve_{\bh^2}^\alpha, \text{ for any }
 (\bu, \bv) \in \bve_i\times \bbe_i, \label{BIL-B1-D}
\end{align}
for appropriate values of $i$.
%
% Let $(Z,\mathcal{Z})$ be a separable metric space and let
%  $\nu$ be a $\sigma$-finite positive measure on it. Suppose that $\mathfrak{P}=(\Omega,\mathcal{F},\FF,\mathbb{P})$ is a filtered
% probability space, where $\FF=(\mathcal{F}_t)_{t\geq 0}$ is a filtration,
% and $\eta: \Omega\times \mathcal{B}(\mathbb{R}_+)\times\mathcal{Z}\to
% \bar{\mathbb{N}}$ is a time homogeneous Poisson random measure with
% the intensity measure $\nu$ defined over the filtered probability
% space $\mathfrak{P}$. We will denote by $\tilde\eta=\eta-\gamma$
% the compensated Poisson random measure associated to $\eta$ where
% the compensator $\gamma$ is given by
% $$
% \mathcal{B}(\mathbb{R}_+)\times \mathcal{Z}\ni (A,I)\mapsto  \gamma(A,I)=\nu(A)\lambda(I)\in\mathbb{R} _ +.
% $$
% This will be fixed in the whole section.
\subsection{Stochastic hydrodynamical systems}
 In this subsection we use exactly the same notations as used in
 \cite{Chueshov}.
\subsubsection{Stochastic Navier-Stokes Equations}
Let  $\CO$ be a bounded, open and simply connected domain  of $\RR^n$, $n=2,3$. The boundary $\partial \CO$ of $\CO$ is assumed to be smooth. Let $(Z, \CZ, \nu)$ be a measure space where the $\nu$ is a $\sigma$-finite, positive measure
and $\tilde{\eta}$ be a
 compensated Poisson random measure having intensity measure $\nu$ defined on filtered complete probability space $\mathfrak{P}=(\Omega,\mathcal{F},\FF,\mathbb{P})$,
 where the filtration $\FF=(\mathcal{F}_t)_{t\geq 0}$ satisfies the usual conditions.
We consider the Navier-Stokes equation with
the Dirichlet (no-slip) boundary conditions:
\begin{equation} \label{1.1.0}
du +\Big[-\kappa \Delta u + u\nabla u + \nabla p \Big]dt=\int_Z \tilde{G}(t,\bu(t),z) \tilde{\eta}(dz,dt) , \quad \mbox{\rm div}\, u=0
~~\mbox{ in }~~D,\qquad u=0\quad\mbox{on}\quad \partial \CO,
\end{equation}
where $u= (u^1(x,t), u^2(x,t))$ is the velocity of a fluid,   $p(x,t)$ is the pressure,
 $\kappa$ the kinematic viscosity. Here $\int_Z \tilde{G}(t,u(t),z) \tilde{\eta}(dz,dt) $ represents a state-dependent random external forcing of jump type.

 Let $\h=\h_1$, $\bve=\bve_1$ and $\bbe=\bbe_1$ where the hilbert spaces $\h_i$, $\bve_i$ and $\bbe_i$ are defined as in Eq. \eqref{STOKES} of Subsection \ref{notation}.
The norms of $\h$, $\ve$ and $\bbe$ are denoted by $\lvert\cdot\rvert$, $\lve \cdot \rve$, $\lve \cdot \rve_\ast$, respectively.

 Let $A=\rA_1$ and $B=B_1$ be the linear and bilinear maps
 defined in Subsection \ref{notation}. We also set $N=A$. Note that in this case  $N$ is self-adjoint and $N\in \mathcal{L}(\bbe, \h)\cap \mathcal{L}(\bve, \bve^\ast)$.
%
%   Let $\tilde{\eta}$ be a
%  compensated Poisson random measure having intensity measure $\nu$ defined on filtered complete probability space $\mathfrak{P}=(\Omega,\mathcal{F},\FF,\mathbb{P})$,
%  where the filtration $\FF=(\mathcal{F}_t)_{t\geq 0}$ satisfies the usual conditions.
%
%  We assume that there exists a function $\tilde{G}$ such that $ f dt= \int_Z \tilde{G}(t,\bu(t),z) \tilde{\eta}(dz,dt)$.
We suppose that
 $\tilde{G}$ satisfies the following sets of conditions.
 \begin{assum}\label{assum-G-Benard}
 We assume that $\tilde{G}$ maps $\ve$ into $L^{2p}(Z,\nu, \ve)$ and there exists a constant $\ell_p>0$ such that
\begin{equation}\label{eqn-local Lipschitz-G-BEN}
\lVert \tilde{G}(x)-\tilde{G}(y)\rVert^{2p}_{L^{2p}(Z,\nu, \ve)}\le \ell^p_p \rve x-y\rve^{2p},
\end{equation}
for any $x, y\in \ve$ and $p=1,2$.

Note that this implies in particular that there exists a constant $\tilde{\ell}_p>0$ such that
\begin{equation}\label{Lipschitz-G-2-BEN}
\lVert \tilde{G}(x)\rVert^{2p}_{L^{2p}(Z,\nu, \h)}\le \tilde{\ell}^p_p(1+ \rve x\rve^{2p}),
\end{equation}
for any $x\in \ve$ and $p=1,2$ .
\end{assum}

 By setting $R\equiv 0$ and projecting on the space of divergence free vector fields the system \eqref{1.1.0} can be rewritten in the following abstract
 form
 \begin{equation}
 \begin{split}\label{Hydro-1}
  & d \bu+ [A\bu+B(\bu,\bu) +R(\bu)] dt = \int_Z \tilde{G}(t,\bu(t),z) \tilde{\eta}(dz,dt),\\
  & \bu(0)=\xi,
 \end{split}
 \end{equation}

\begin{thm}\label{Thm-NS}
 The stochastic Navier-Stokes problem \eqref{Hydro-1} admits a local maximal strong solution which is global if $n=2$.
\end{thm}
\begin{Rem}
 This theorem remains true in the case $\CO$ being a general unbounded domain. For the proof it is sufficient to take $A=A_1+\Id$, $R(\bu):=-\bu$ and argue
 as in the case of
 bounded domain.
 \end{Rem}

\begin{proof}
The existence and uniqueness of a maximal local solution will follow from  Theorem \ref{MAX-FULL-EQ} if we are able to prove that $F(\bu)=B(\bu,\bu)$ satisfies \eqref{eqn-local Lipschitz-F}.
%Since $R$ is a bounded linear map, it follows from Remark \ref{REM-GLO} that it is sufficient to check \eqref{eqn-local Lipschitz-F} for $B$.
But from \eqref{BIL-B1-A} we deduce that there exists
 $C>0$
such that for
\begin{equation*}%\label{eqn-local Lipschitz-F}
\lvert B(y)-B(x) \lvert \leq C \Big[ \lve y-x\lve \lve
y\lve^{1-\frac n4} \Vert y\Vert_\ast^\frac n4 + \Vert y-x\Vert_\ast^\frac n4
\lve y-x\lve^{1-\frac n4} \lve x\lve\Big],
\end{equation*}
for any $x, y\in \be$. This means that $B$ satisfies \eqref{eqn-local Lipschitz-F} with $p=1$ and $\alpha=\frac n4$.
Since $\alpha=\frac 34 \notin [0,\frac12]$ for $n=3$, the solution is only maximal. For $n=2$ we have $\alpha=\frac 12$ and $\langle B(\bu,\bu), \bu\rangle=0$. So thanks to Remark \ref{REM-GLO}, we only need
to check that \eqref{Cond-GLO} is verified by $B$. But this will follow from \eqref{BIL-B1-C}.
\end{proof}%%%%%%%%%%%%%%%%%%%%%%%%
%%%%%%%%%%%%%%%%%%%%%%%%
%%%%%%%%%%%%%%%%%%%%%%%%
 \subsubsection{Magnetohydrodynamic equations}\label{MHD}
Let $\CO\subset \RR^n$, $n=2,3$ be a simply connected, possibly unbounded domain. As above we assume that $\CO$ has a smooth boundary $\partial \CO$. Let $(Z_i, \CZ_i, \nu_i)$, $i=1,2$ be two measure spaces where the measures $\nu_i$ are $\sigma$-finite and positive.
We consider two mutually independent compensated Poisson random measures $\tilde{\eta}_i$ with
intensity measure $\nu_i$ defined on a complete filtered probability space $\mathfrak{P}=(\Omega,\mathcal{F},\FF,\mathbb{P})$.
We consider the magneto-hydrodynamic (MHD)  equations
  \begin{equation}\label{1.1u}
   du+[-\Delta u+
   u \nabla u]dt= [-\nabla \left(p+\frac{1}2 |b|^2\right)
   + b \nabla b]dt+
   \int_{Z_1} \tilde{f}(t, u(t), b(t),z_1)  \tilde{\eta}_1(dz_1, dt),
\end{equation}
\begin{equation}\label{1.1b}
   db+[-\nu_2\Delta b+
   u \nabla b]dt=
   [b\nabla u] dt +
   \int_{Z_2} \tilde{g}(t, u(t), b(t),z_2)  \tilde{\eta}_2(dz_2, dt),
\end{equation}
\begin{equation}\label{1.2}
   \di u=0, \quad  \di b =0 \quad
\end{equation}
where
$u=(u^{(1)}(x,t),u^{(2)}(x,t), u^{(3)}(x,t))$  and $b=(b^{(1)}(x,t),b^{(2)}(x,t), b^{(3)}(x,t))$
denote  velocity and magnetic fields,
$p(x,t)$ is a scalar pressure.
We consider the following
boundary conditions
\begin{equation}\label{1.1bc}
u=0, \quad   b\, \cdot\, n=0, \quad \text{curl }b\times n=0
\quad {\rm on}~~ \partial \CO
\end{equation}
The terms $\int_{Z_1} \tilde{f}(t, u(t), b(t),z_1)\tilde{\eta}_1(dz_1, dt)$ and
 $\int_{Z_2} \tilde{g}(t, u(t), b(t), z_2) \tilde{\eta}_2(dz_2, dt),$ represent random external volume
forces  and  the curl of random external current applied to the fluid.
We refer to \cite{LadSol-mhd60}, \cite{DL-mhd72} and \cite{SeTe}
for the mathematical theory for the MHD equations.

Let $\h=\h_1\times \h_2$, $\ve=\ve_1\times \ve_2$ and $\bbe=\bbe_1\times \bbe_2$. We define a bilinear map $B(\cdot, \cdot)$ on $\bve\times \bbe$
by \begin{eqnarray*}
\langle B(z_1,z_2), z_3\rangle & = &\langle B_1(u_1,u_2), u_3\rangle
-\langle B_1(b_1,b_2), u_3\rangle
\\ & &
+\, \langle B_1(u_1,b_2), b_3\rangle-
\langle B_1(b_1,u_2), b_3\rangle,
\end{eqnarray*}
for $z_1=(u_1, b_1)\in \bve$, $z_2=(u_2,b_2)\in \bbe$ and $z_3=(u_3,b_3)\in \h$.
We also set
\begin{equation*}
 A z= \begin{pmatrix}
  \Id+A_1 & 0\\
  0 & \Id+A_2
 \end{pmatrix}
 \begin{pmatrix}
  u\\  b
 \end{pmatrix}
\end{equation*}
for $z=(u, b)\in \bbe .$

We set $\bu:=(u, b)$ and
\begin{align*}
\int_Z \tilde{G}(t,\bu(t),z) \tilde{\eta}(dz,dt):=
\begin{pmatrix}
   \int_{Z_1} \tilde{f}(t, \bu(t),z_1) \tilde{\eta}_1(dz_1, dt) \\
  \int_{Z_2} \tilde{g}(t,\bu(t)), z_2) \tilde{\eta}_2(dz_2, dt)
 \end{pmatrix}.
\end{align*}
 We assume that $\tilde{f}$, $\tilde{g} $ are chosen in such a way that $\tilde{G}$ maps $\ve$ into $L^{2p}(Z,\nu, \ve)$
 and satisfies Assumption \ref{assum-G-Benard}.

 By setting $R\equiv -\Id$ and projecting on $\h$ we can see that \eqref{1.1u}-\eqref{1.1b} can be rewritten in the form
 \eqref{Hydro-1}. Now, by choosing $N=A$ we can show by arguing as in Theorem \ref{Thm-NS} that the stochastic Magnetohydrodynamic equations \eqref{1.1u}-\eqref{1.1b} has a local
 maximal solution which is global if the dimension $n=2$.

 \subsubsection{Magnetic B\'{e}nard problem.}
Let $\CO=(0, l) \times (0, 1)$ be a rectangular   domain  in the
vertical plane,  $(e_1, e_2)$ the standard basis in
 $\RR^2$. Let $(Z_i, \CZ_i, \nu_i)$, $i=1,2,3$ be three measure spaces where the measures $\nu_i$ are $\sigma$-finite and positive.
We consider three mutually independent compensated Poisson random measures $\tilde{\eta}_i$ with
intensity measure $\nu_i$ defined on a complete filtered probability space $\mathfrak{P}=(\Omega,\mathcal{F},\FF,\mathbb{P})$.

 We consider the equations
\begin{eqnarray*}
d u + [u \dn u-\kappa_1 \D u + \nabla \left(p+\frac{s}2 |b|^2\right)
   -s b \nabla b]dt &=&  \th e_2  dt + \int_{Z_1} \tilde{f}(t, u(t), \th(t),b(t),z_1)  \tilde{\eta}_1(dz_1, dt),\\
\quad \di u  & =&  0, %\label{1.1mbp}
\\
d \th +[u \dn \th -u^{(2)} -\k \D \th ]dt&=&  \;\int_{Z_2} \tilde{g}(t, u(t),\th(t), b(t),z_2) \tilde{\eta}_2(dz_2, dt), %\label{eqn3mbp}
\\
  d b+[-\kappa_2\Delta b+
   u \nabla b -    b\nabla u]dt &=&
   \int_{Z_3} \tilde{h}(t, u(t),\th(t), b(t),z_3) \tilde{\eta}_2(dz_3, dt),\\ \di b   &=& 0, %\label{eqn3mag}
\end{eqnarray*}
with boundary conditions
\begin{eqnarray*}
u =0, \;\; \th=0, \;\; b^{(2)}=0,\;\; \partial_2 b^{(1)}=0 \;\;  \mbox{on}\;\;x^{(2)}=0\; \mbox{and}\;x^{(2)}=1,   \\
u, p, \th, b, u_{x^{(1)}}, \th_{x^{(1)}}, b_{x^{(1)}} \; \text{ are periodic in}\; x^{(1)} \;
\mbox{with period}\; l.
\end{eqnarray*}
This is  the  Boussinesq model coupled with magnetic field (see \cite{GaPa}) with stochastic perturbations.
Throughout we assume that $\kappa_1=\kappa_2=s=1$.
Let $\h= \h_3\times \h_4\times \h_5$, $\bve=\bve_3 \times \bve_4\times \bve_5$ and $\bbe=\bbe_3\times \bbe_4\times \bbe_5$.

We define a bilinear map $B(\cdot, \cdot)$ on $\bve\times \bbe$
by \begin{eqnarray*}
\langle B(z_1,z_2), z_3\rangle & = &\langle B_1(u_1,u_2), u_3\rangle
-\langle B_1(b_1,b_2), u_3\rangle
\\ & &
+\, \langle B_1(u_1,b_2), b_3\rangle-
\langle B_1(b_1,u_2), b_3\rangle
+ \langle B_2(u_1, \th_2,), \th_3 \rangle,
\end{eqnarray*}
for $z_1=(u_1,\th_1,b_1)\in \bve$, $z_2=(u_2,\th_2,b_2)\in \bbe$ and $z_3=(u_3,\th_3,b_3)\in \h$.
Using the notations in \eqref{STOKES}, we set
\begin{equation*}
 A z= \begin{pmatrix}
  A_3 & 0& 0\\
  0 & A_4 & 0\\
  0& 0& A_5
 \end{pmatrix}
 \begin{pmatrix}
  u\\ \theta\\ b
 \end{pmatrix}
\end{equation*}
for $z=(u,\th, b)\in E .$

We also set $R(u,\theta, b)= - (\theta e_2\, ,\,
u^{(2)}, 0)$ and $N=A$. Note that in this case $R \in \mathcal{L}(\h, \h)$ and $N\in \mathcal{L}(\bbe, \h)\cap \mathcal{L}(\bve, \bve^\ast)$.
% For $\bu=(u,\theta, b)\in \bbe$, let $F(\bu)= B(\bu,\bu) +R(\bu)$. Now we can rewrite the Magnetic B\'{e}nard problem in the following way
% \begin{equation}
%  \begin{split}\label{Hydro-0}
%   & d \bu+ [A\bu+B(\bu,\bu) +R(\bu)]dt = G dt,\\
%   & \bu(0)=\xi,
%  \end{split}
%  \end{equation}
% where $G=(f,g,h)$.

We set $\bu:=(u, \th, b)$ and
\begin{align*}
 \int_Z \tilde{G}(t,\bu(t),z) \tilde{\eta}(dz,dt):=
\begin{pmatrix}
  \int_Z \tilde{f}(t, \bu(t),z_1) \tilde{\eta}_1(dz_1, dt) \\
  \int_Z \tilde{g}(t, \bu(t),z_2) \tilde{\eta}_2(dz_2, dt)\\
  \int_Z \tilde{h}(t, \bu(t),z_3) \tilde{\eta}_3(dz_3, dt).
 \end{pmatrix}
\end{align*}
We assume that $\tilde{f}, \tilde{g}, \tilde{h}$ are chosen such that $\tilde{G}$ verifies
 Assumption \ref{assum-G-Benard}.
 With these notations we can put the stochastic Magnetic B\'{e}nard problem into the abstract stochastic evolution equation
  \eqref{Hydro-1}.
%  satisfying some conditions given later.
% Hence the stochastic model for the Magnetic B\'{e}nard problem is of
% \begin{equation}
%  \begin{split}\label{Hydro-1}
%   & d \bu+ [A\bu+B(\bu,\bu) +R(\bu)] dt = \int_Z \tilde{G}(t,\bu(t),z) \tilde{\eta}(dz,dt),\\
%   & \bu(0)=\xi,
%  \end{split}
%  \end{equation}
%  Throughout this section we choose the function $f, g, h$ in such way that the following set of constraints are met.
% \begin{assum}\label{assum-G-Benard}
%  We choose $f$, $g$ and $h$ such that $G$ maps $\ve$ into $L^{2p}(Z,\nu, \ve)$ and there exists a constant $\ell_p>0$ such that
% \begin{equation}\label{eqn-local Lipschitz-G-BEN}
% \lVert G(x)-G(y)\rVert^{2p}_{L^{2p}(Z,\nu, \ve)}\le \ell^p_p \rve x-y\rve^{2p},
% \end{equation}
% for any $x, y\in \ve$ and $p=1,2$.
%
% Note that this implies in particular that there exists a constant $\tilde{\ell}_p>0$ such that
% \begin{equation}\label{Lipschitz-G-2-BEN}
% \lVert G(x)\rVert^{2p}_{L^{2p}(Z,\nu, \h)}\le \tilde{\ell}^p_p(1+ \rve x\rve^{2p}),
% \end{equation}
% for any $x\in \ve$ and $p=1,2$ .
% \end{assum}
\begin{thm}\label{Thm-Benard}
 The stochastic Magnetic B\'{e}nard problem \eqref{Hydro-1} admits a unique global strong solution.
\end{thm}
\begin{proof}
The maximal local solution will follow from  Theorem \ref{MAX-FULL-EQ}
if we are able to prove that $F(\bu)=B(\bu,\bu)+R(\bu)$ satisfies \eqref{eqn-local Lipschitz-F}.
Since $R$ is a bounded linear map, it follows from Remark \ref{REM-GLO} that it is sufficient to check \eqref{eqn-local Lipschitz-F} for $B$.
But from \eqref{BIL-B1-A} and \eqref{BIL-B1-B} we deduce that there exists
 $C>0$
such that for
\begin{equation*}%\label{eqn-local Lipschitz-F}
\lvert B(y)-B(x) \lvert \leq C \Big[ \lve y-x\lve \lve
y\lve^{1-\frac n4} \Vert y\Vert_\ast^\frac n4 + \Vert y-x\Vert_\ast^\frac n4
\lve y-x\lve^{1-\frac n4} \lve x\lve\Big],
\end{equation*}
for any $x, y\in \be$. This means that $B$ satisfies \eqref{eqn-local Lipschitz-F} with $p=1$ and $\alpha=\frac n4$.
Since $n=2$ we have $\alpha=\frac 12$ and $\langle B(\bu,\bu), \bu\rangle=0$. So thanks to Remark \ref{REM-GLO}, we only need
to check that \eqref{Cond-GLO} is verified by $B$. But this will follow from \eqref{BIL-B1-C} and \eqref{BIL-B1-D}.
\end{proof}

 %%%%%%%%%%%%%%%%%%%%%%%%%%%%%%%
 %%%%%%%%%%%%%%%%%%%%%%%%%%%%%%
 \subsubsection{Boussinesq model for the B\'enard convection}

Let $\CO$ be a (possibly) domain of $\mathbb{R}^n$, $n=2,3$,  $\{e_i, \ldots, e_n\}$ a standard basis in
 $\RR^n$  and $x=(x^{(1)},\ldots,x^{(n)})$ an element of $\RR^n$. We assume that $\CO$ has a smooth boundary $\partial \CO$. Let $(Z_i, \CZ_i, \nu_i)$, $i=1,2$ be two
 measure spaces where the measures $\nu_i$ are $\sigma$-finite and positive.
We consider two mutually independent compensated Poisson random measures $\tilde{\eta}_i$ with
intensity measure $\nu_i$ defined on a complete filtered probability space $\mathfrak{P}=(\Omega,\mathcal{F},\FF,\mathbb{P})$.

 Let us consider the B\'{e}nard convection problem
(see e.g. \cite{Foias} and the references therein) given by the following system
\begin{eqnarray}
d u + [u \nabla u- \D u + \nabla p ]dt&=&  \th e_n  dt+ \int_{Z_1} \tilde{f}(t, u(t), b(t),z_1) \tilde{\eta}_1(dz_1, dt),
\quad \di u   = 0, \label{1.1}\\
 d\th +[u \nabla \th -u^{(n)} -\D \th ]dt&=&
\int_{Z_2} \tilde{g}(t, u(t), \th(t),z_2) \tilde{\eta}_2(dz_2, dt),\label{eqn3}
\end{eqnarray}
with boundary conditions
\begin{eqnarray*}
u =0\;\;  \& \;\; \th=0 \;\; \mbox{on } \partial \CO.
\end{eqnarray*}
 Here $p(x,t)$ is the pressure field,   $\int_{Z_1} \tilde{f}(t, u(t), b(t), z_1) \tilde{\eta}_1(dz_1, dt)$,
 $\int_{Z_2} \tilde{g}(t, u(t), b(t), z_2) \tilde{\eta}_2(dz_2, dt)$   represent two random external forces,
$ u=(u^{(1)}(x,t),\ldots ,u^{(n)}(x,t))$ is the velocity field and  $\th=\th(x,t)$
is the temperature field.

We set $\h=\h_3\times \h_4$, $\ve=\ve_3\times \ve_4$, $\bbe=\bbe_3\times \bbe_4$. Following the notations given in \eqref{STOKES} we define
\begin{equation*}
 A z= \begin{pmatrix}
  A_3 & 0\\
  0 & A_4
 \end{pmatrix}
 \begin{pmatrix}
  u\\ \theta
 \end{pmatrix}
\end{equation*}
for $z=(u,\th)\in E .$ We define a bilinear map $B(\cdot, \cdot)$ on $\bve\times \bbe$
by \begin{eqnarray*}
\langle B(z_1,z_2), z_3\rangle & = &\langle B_1(u_1,u_2), u_3\rangle
+ \langle B_2(u_1, \th_2,), \th_3 \rangle,
\end{eqnarray*}
for $z_1=(u_1,\th_1)\in \bve$, $z_2=(u_2,\th_2)\in \bbe$ and $z_3=(u_3,\th_3)\in \h$.
We also put $R(u,\theta, b)= - (\theta e_2\, ,\,
u^{(n)})$ and $N=A$.

As before we set $\bu:=(u, \th)$ and
\begin{align*}
\int_Z \tilde{G}(t,\bu(t),z) \tilde{\eta}(dz,dt):=
\begin{pmatrix}
   \int_{Z_1} \tilde{f}(t, u(t), b(t), z_1)  \tilde{\eta}_1(dz_1, dt) \\
  \int_{Z_2} \tilde{g}(t, u(t), b(t), z_2)  \tilde{\eta}_2(dz_2, dt)
 \end{pmatrix}.
\end{align*}
 We assume that $\tilde{f}$, $\tilde{g}$ are chosen in such a way that $\tilde{G}$ maps $\ve$ into $L^{2p}(Z,\nu, \ve)$
 and satisfies Assumption \ref{assum-G-Benard}.

By Arguing as in the case of Navier-Stokes equations, Magnetic B\'{e}nard problem and MHD equations
we can show that if the random external force satisfies
 Assumption \ref{assum-G-Benard}, then the Boussinesq model for the B\'enard convection admits a
 unique maximal strong solution which is global is $n=2$.

\subsection{Shell models of turbulence}
Here, we use again the same notations as used in
 \cite{Chueshov}.
Let $H$ be a set of all sequences $u=(u_1, u_2,\ldots)$ of complex
numbers such that $\sum_n |u_n|^2<\infty$. We consider $H$ as a
\emph{real} Hilbert space endowed  with the inner product
$(\cdot,\cdot)$ and the norm $|\cdot|$ of the form
\[
(u,v)={\rm Re}\,\sum_{n=1}^\infty u_n v_n^*,\quad
|u|^2 =\sum_{n=1}^\infty |u_n|^2,
\]
where $v_n^*$ denotes the complex conjugate of $v_n$. In this space $H$ we
consider the evolution equation \eqref{Hydro-1} with $R=0$ and with linear operator $A$
and bilinear mapping $B$ defined by the formulas
\[
(Au)_n =\nu k_n^2 u_u,\quad n=1,2,\ldots,\qquad
D(A)=\left\{ u\in H\, :\; \sum_{n=1}^\infty k_n^4 |u_n|^2<\infty\right\},
\]
where $\nu>0$, $k_n=k_0\mu^n$ with $k_0>0$ and $\mu>1$, and
\[
\left[B(u,v)\right]_n=-i\left( a k_{n+1} u_{n+1}^* v_{n+2}^*
+b k_{n} u_{n-1}^* v_{n+1}^* -a k_{n-1} u_{n-1}^* v_{n-2}^*
-b k_{n-1} u_{n-2}^* v_{n-1}^*
\right)
\]
for $n=1,2,\ldots$, where $a$ and $b$ are real numbers (here above we also assume that
$ u_{-1}= u_{0}=v_{-1}= v_{0}=0$). This choice  of $A$ and $B$ corresponds to the
so-called GOY-model (see, e.g., \cite{OY89}).
If we take
\[
\left[B(u,v)\right]_n=-i\left( a k_{n+1} u_{n+1}^* v_{n+2}
+b k_{n} u_{n-1}^* v_{n+1} +a k_{n-1} u_{n-1} v_{n-2}
+b k_{n-1} u_{n-2} v_{n-1}
\right),
\]
then we obtain the Sabra shell model introduced in \cite{LPPPV98}.
\par
One can easily  show (see \cite{BBBF06} for the GOY model and \cite{CLT06}
for the Sabra model) that
the trilinear form
\[
\langle B(u,v), w\rangle\equiv {\rm Re}\, \sum_{n=1}^\infty [B(u,v)]_n\,  w_n^*
\]
 satisfies the inequality
\[
\left|\langle B(u,v), w\rangle\right|\le C |u||A^{1/2} v| |w|,\quad
\forall  u,w\in H, \quad \forall v\in D(A^{1/2}).
\]
Hence taking  $\h=H$, $(\ve, \lve\cdot \lve)=(D(A^\frac12), \lvert A^\frac12 \cdot \lvert)$ and
$(\bbe, \lve \cdot \rve_\ast):=(D(A), \lvert A\cdot \rvert)$ we infer that
the nonlinear term for the shell models satisfies Assumption \ref{assum-F} with $\alpha =0$ and $p=1$.
By Arguing as before we can show that if the random external force satisfies
 Assumption \ref{assum-G-Benard}, then stochastic shell models admits a
 unique global strong solution.
% In addition $\langle B(u,v),v\rangle=0$ for any $u\in H$, $v\in v\in D(A^{1/2})$.
% Thus by Remark~\ref{re:1}(1)
% the condition ({\bf C1}) holds with $\HH= Dom(A^{s})$ for any choice
% of $s\in [0,1/4]$.
\medskip\par
Let us consider the following dyadic model (see, e.g., \cite{KP05}
and the references therein) \begin{equation}\label{dyadic-m}
\partial_t u_n+\nu\lambda^{2\alpha n} u_n-\lambda^{n} u^2_{n-1}+
\lambda^{n+1} u_n u_{n+1}=f_n, \quad n=1,2,\ldots,
\end{equation}
where  $\nu, \alpha>0$, $\lambda>1$, $u_0=0$. By setting
$[B(u,v)]_n=-\lambda^n u_{n-1} v_{n-1}+ \lambda^{n+1}\, u_n\, v_{n+1}$ and
$(Au)_n=\nu\, \lambda^{2\alpha n}\, u_n$, it is not difficult to show that the system \eqref{dyadic-m}
falls also in the framework of the shell models of turbulence provided that $\alpha\ge 1/2$.

\subsection{3D Leray $\alpha$-model for Navier-Stokes equations}
As in the previous subsections, we use the same notations as used
in
 \cite{Chueshov}.

In a bounded 3D domain $\CO$ we consider the following
equations:
\begin{align}
& \partial_t u - \Delta u + v\nabla u + \nabla p =f ,\label{1.1.0-3da}
\\
& (1-\alpha \Delta)v=u,\quad \mbox{\rm div}\, u=0,\quad \mbox{\rm div}\, v=0
~~\mbox{ in }~~\CO, \label{1.1.0-3db} \\
&v=u=0\quad\mbox{on}\quad \partial \CO. \label{1.1.0-3dc}
\end{align}
where $u= (u^{(1)}, u^{(2)}, u^{(3)})$ and $v= (v^{(1)}, v^{(2)}, v^{(3)})$  are unknown fields,
 $p(x,t)$ is the pressure. We refer to \cite{titi2,titi1} and references for results related to \eqref{1.1.0-3da}-\eqref{1.1.0-3dc}.

 Let $\h=\h_1$, $\ve=\ve_1$ and $\bbe=\bbe_1$ be the Hilbert spaces defined in Subsection \ref{notation}.
 Set $A=A_1$, $G_\alpha=\left(Id+\alpha A\right)^{-1}$ and define a bilineear mapping $B(\cdot, \cdot)$ on $\ve\times \bbe$ by setting
 $$B(u,v)=B_1(G_\alpha u, v),$$
 for any $u\in \ve$ and $v\in \bbe$.

 Arguing as in \cite[Subsubsection 2.1.5]{Chueshov} we can show that there exists $C>0$ such that
 \begin{equation}\label{LER-3D}
  \lvert B(u,v)\rvert \le C \lve u\rve_{L^3} \lve \nabla v\rve_{L^3},
 \end{equation}
for any $u\in \el^3$ and $v\in \mathbb{W}^{1,3}$. %Using the continuous embedding $\el^3 \subset \h^1 $ and the folloowing
Recall that in three dimensional case we have the following Gagliardo-Nirenberg
inequality
\begin{equation}\label{GAG-NIR-3D}
 \lve u\rve_{L^3}\le c \lvert u\rvert^\frac12 \lve u \rve^\frac12_{\bh^1}, \quad u\in \bh^1.
\end{equation}
Now using this inequality and the continuous embedding $\bh^1\subset \el^3$

we infer from  \eqref{LER-3D} that
\begin{align}
 \lvert B(u,v)\rvert \le C \lve u\rve_\ve \lve v\rve_\ve^\frac12 \lve v\rve_\bbe,\label{LER-BIL-1}\\
 \lvert B(u,v)\rvert \le C \lvert u\rvert_\h^\frac12 \lve u\rve_\ve^\frac12 \lve v\rve_\ve^\frac12 \lve v\rve_\bbe, \label{LER-BIL-2}
\end{align}
for any $u\in \ve$, $v\in \bbe$.

Now we set $R\equiv 0$ and $N=A$. Thanks to \eqref{LER-BIL-1}-\eqref{LER-BIL-2} we see that the nonlinear term
for the 3D Leray $\alpha$-model for Navier-Stokes equations satisfies the assumptions of Theorem \ref{GLO-FULL-EQ} with $\alpha=\frac12$
and $p=1$. Therefore we can argue as in the case of 2D stochastic Navier-Stokes equations and show that the stochastic 3D Leray
$\alpha$-model for Navier-Stokes equations admits a global solution if the random external force satisfies
 Assumption \ref{assum-G-Benard}.

\appendix

\section{Existence of solution to the linear SPDE \eqref{lin-eq-1} }

Throughout this appendix we assume that the separable Hilbert spaces $E, V, H$ are defined as before.

 Let $(Z,\mathcal{Z})$ be a separable metric space and let
 $\nu$ be a $\sigma$-finite positive measure on it. Let $\eta: \Omega\times \mathcal{B}(\mathbb{R}_+)\times\mathcal{Z}\to
\bar{\mathbb{N}}$ is a time homogeneous Poisson random measure with
the intensity measure $\nu$. We will denote by $\tilde\eta=\eta-\gamma$
the compensated Poisson random measure associated to $\eta$ where
the compensator $\gamma$ is given by
$$
\mathcal{B}(\mathbb{R}_+)\times \mathcal{Z}\ni (A,I)\mapsto  \gamma(A,I)=\nu(A)\lambda(I)\in\mathbb{R} _ +.
$$
Let $\phi\in M^2(0,T;\h)$ and $\psi\in M^2(0,T;L^4(Z,\nu; \ve))$. We will show in the next theorem that
the following linear SPDEs ( which is \eqref{lin-eq-1}) has a unique solution
\begin{equation}\label{lin-eq-2}
 \begin{cases}
  d\bu(t)+[A\bu(t)+\phi(t)]dt=\int_Z \psi(t,z ) \wie(dz,dt), \,\, t\in[0,T],\\
  \bu(0)=\bu_0.
 \end{cases}
\end{equation}

\begin{thm}\label{Appendix-LIN}
Let $A$, $N$ be as in Assumption \ref{lin-terms}, $\phi\in M^2(0,T;\h)$, $\psi\in M^2(0,T;L^2(Z,\nu; \ve))$. Let $\bu_0$ be a $\ve$-valued $\mathcal{F}_0$-measurable
random variable satisfying $\mathbb{E}\lvert \bu_0\rvert^2<\infty$.
Then there exists a unique progressively measurable process $\bu$ taking values in $\ve$ such that almost surely
 \begin{equation}\label{lin-eq-3}
  \langle \bu(t), w\rangle +\int_0^t \langle A\bu(s)+\phi(s), w\rangle ds= \langle \bu_0, w\rangle +\int_0^t \int_Z \langle \psi(s,z ), w\rangle \wie(dz,ds),
 \end{equation}
 for all $t\in [0,T]$ and $w\in \h$. Moreover,
 $ \bu\in L^\infty(0,T;\ve)\cap L^2(0,T;\be)\cap \mathbf{D}(0,T;\ve) $ with probability 1.
\end{thm}
\begin{proof}
We will use the Picard method as presented in \cite[Chapitre 3, Section 1]{Pardoux-75}. Throughout this proof we set $$\Lve N\Lve:=\max\left(\noe, \nov \right).$$
For positive integer $n$ we define a sequence $\{\bu^{[n]}(t), t\in[0,T]\}$ of stochastic processes as follows
\begin{equation*}
\begin{cases}
 \bu^{[1]}(t)=\bu_0,\\
 \bu^{[n+1]}(t)=\bu_0-\int_0^t[A\bu^{[n]}(s)+\phi(s)]ds+\int_0^t\int_Z \psi(s,z ) \wie(dz,ds), \,\, t\in [0,T],\,\, n\ge 2.
\end{cases}
 \end{equation*}
Thanks to our assumption and \cite[Theorem 2]{Gyongy+Krylov} the strochastic processe
$$ \bu^{[2]}(t)=\bu_0-\int_0^t[A\bu^{[1]}(s)+\phi(s)]ds +\int_0^t\int_Z \psi(s,z ) \wie(dz,ds) $$ is a well-defined
$\ve$-valued adapted and \cadlag process.
By iterating this definition we see that for each $n\ge 2$ $\bu^{[n]}$ is also a well-defined $\ve$-valued adapted and \cadlag process.

Now we will show that the sequence $\bu^{[n]}$ is converging in appropriate topolgy to the solution $\bu$ of \eqref{lin-eq-2}. In fact we will show that
$\bu^{[n]}\in L^2(\Omega;L^\infty(0,T;\ve))$ is a Cauchy sequence.
For this aim define $\Phi^n(t)=\EE\sup_{s \in[0,t] }\lve \bu^{[n+1]}(s)-\bu^{[n]}(s)\rve^2$ for all $n\ge1$. We have
\begin{equation*}
 \bu^{[n+1]}(t)-\bu^{[n]}(t)=-\int_0^t A(\bu^{[n]}(s)-\bu^{[n-1]})ds,
\end{equation*}
for any $t\in [0,T]$ and $n\ge 2$.
Mutliplying this equation by $N(\bu^{[n+1]}-\bu^{[n]})$, using Assumption \ref{lin-terms} and the Cauchy inequality with arbitrary $\eps>0$ we infer that
\begin{equation*}
 (C_N-\eps)\sup_{s\in [0,t] }\lve \bu^{[n+1]}(s)-\bu^{[n]}(s)\rve^2\le \frac{\Lve N \Lve^2 \Lve A\Lve^2}{4 \eps} \int_0^t \lve \bu^{[n]}(s)-\bu^{[n-1]}(s)\rve^2 ds.
\end{equation*}
Choosing $\eps=C_N/2$ taking the mathematical expectation to both side of the last estimate implies
\begin{equation}\label{To-Iter}
 \Phi^n(t)\le \frac{\Lve N \Lve^2 \Lve A\Lve^2}{2 C^2_N} \int_0^t \Phi^{n-1}(s) ds.
\end{equation}
As in \cite{Pardoux-75} we iterate \eqref{To-Iter} and obtain
\begin{equation*}
 \Phi^n(t)\le \left(\frac{\Lve N \Lve^2 \Lve A\Lve^2}{2 C^2_N} \right)^n  \frac{1}{ n!}\Phi^1(t),
\end{equation*}
from which we deduce that $(\bu^{[n]}; n\ge 1)$ forms a Cauchy sequence in $ L^2(\Omega;L^\infty(0,T;\ve))$.
Therefore, there exists $\bu\in  L^2(\Omega;L^\infty(0,T;\ve))$ such that
\begin{equation} \label{Conv-1}
 \bu^{[n]} \rightarrow \bu \text{ strongly in }  L^2(\Omega;L^\infty(0,T;\ve)).
\end{equation}
Now, we prove that $\bu^{[n]}$ is bounded in $L^2(\Omega; L^2(0,T;\be))$. For this purpose we apply It\^o formula to $\Psi(u)=\langle u, Nu\rangle$ and use Assumption
\ref{assum-G} to infer that
\begin{equation}\label{APP-INEQ-1}
\begin{split}
 C_N\lve\bu^{[n]}(t\wedge \tau)\rve^2+2 C_A\int_0^{t\wedge \tau} \lve \bu^{[n]}(s)\rve_\ast^2 ds\le
 \int_0^T \Big[\lvert \phi(s)\rvert^2 +\int_Z \langle  N \psi(s,z), \psi(s,z)  \rangle\nu(dz) \Big] ds\\
 + \int_0^{t\wedge\tau}\int_Z  \Big[\langle \psi(s,z), N \bu^{[n]}(s)\rangle +\langle \psi(s,z) , N \psi(s,z) \rangle \Big] \wie(dz,ds)\\
 +\Psi(\bu_0)+
\Lve N\Lve^2 \int_0^T \lvert \bu^{[n]}(s)\rvert^2 ds .
 \end{split}
\end{equation}
where $\tau$ is an arbitrary stopping time localizing the local martingale
$$\int_0^{t}\int_Z  \Big[\langle \psi(s,z), N \bu^{[n]}(s)\rangle +\langle \psi(s,z) , N \psi(s,z) \rangle \Big] \wie(dz,ds).$$
We easily derive from \eqref{APP-INEQ-1} that
\begin{equation*}
\begin{split}
 C_N\lve \bu^{[n]}(t\wedge \tau)\rve^2+2 C_A\int_0^{t\wedge \tau} \lve \bu^{[n]}(s)\rve_\ast^2 ds\le
 \int_0^T \Big[\lvert \phi(s)\rvert^2 +\Lve N\Lve^2 \int_Z \lve \psi(s,z)\rve^2\nu(dz) \Big] ds\\
 + \int_0^{t\wedge\tau}\int_Z  \Big[\langle \psi(s,z), N \bu^{[n]}(s)\rangle +\langle \psi(s,z) , N \psi(s,z) \rangle \Big] \wie(dz,ds)\\
 +\Psi(\bu_0)+
\Lve N\Lve^2 \int_0^T \lve \bu^{[n]}(s)\rve^2 ds .
 \end{split}
\end{equation*}
Since, by the first part of our proof, $\int_0^T \EE \lve \bu^{[n]}(s)\rve^2 ds $ is bounded and $\tau$ is arbitrary, by taking mathematical
expectation to both sides of the last inequality we derive that
there exists $C>0$ such that
\begin{equation*}
\EE\int_0^T \lve \bu^{[n]}(s)\rve_\ast^2 ds\le C.
\end{equation*}
This implies that one can  find a subsequence of $\bu^{[n]}$, which will be denoted wuth the same fashion, such that
\begin{equation}\label{Conv-1-b}
 \bu^{[n]}\rightarrow \bu \text{ weakly in } L^2(\Omega;L^2(0,T;\be)).
\end{equation}
Since, by assumption, $A\in \mathcal{L}(\be,\h)$ it follows from \eqref{Conv-1-b} that
\begin{equation}\label{Conv-2}
 A\bu^{[n]}\rightarrow A\bu \text{ weakly in }  L^2(\Omega;L^2(0,T;\h)).
\end{equation}
Owing to the convergences \eqref{Conv-1} and \eqref{Conv-2} we easily derive that, with probability 1,
$\bu$  satisfies \eqref{lin-eq-3} for all
$t\in [0,T]$ and $w\in \h$.
This means that \eqref{lin-eq-2} holds for all $t\in [0,T]$ and all $w\in \h$ with probability 1.
Since $\bu$ is the limit in $L^2(\Omega;L^\infty(0,T;\ve))$ of a sequence of adapted processes, we infer that $\bu$ is adapted.
Thanks to our assumption and \cite[Theorem 2]{Gyongy+Krylov} the process $\bu$ is \cadlag. Because $\bu$ is adapted and \cadlag it
admits a progressively measurable version which is still denoted with the same symbol.
The proof of our theorem is complete.
\del{
 We will closely follow \cite[Lemma 1.4]{Pardoux}. Let $M(\cdot)=\int_0^{\cdot} \psi(s) dW(s)$. We have $M\in M^2(0,T; H)$ because $\psi\in M^2(0,T;\mathcal{J}_2(K;H))$.
 Hence there exists $\Omega_0$ with $\mathbb{P}(\Omega_0)=1$ such that $AM(\omega)\in L^2(0,T; \h)$ for each $\omega \in \Omega_0$. For each $\omega \in \Omega_0$ we can
 infer from Lemma \ref{Appendix-DET} that there exists a unique
 $v \in L^\infty(0,T;H)\cap L^2(0,T;V)$ satisfying \eqref{DET-LIN} (with $v_0=u_0$ and $f= -\phi-AM$) on $\Omega_0$ for all $t\in [0,T]$. That is, we have just proved that
  there exist $v \in L^\infty(0,T;H)\cap L^2(0,T;V)$ a.s such that
 \begin{equation}\label{DET-LIN-2}
\begin{cases}
 \frac{\partial \bv(t)}{\partial t}+A\bv(t)=-\phi(t)-AM(t), \\
 \bv(0)=u_0,
\end{cases}
\end{equation}
for all $t\in [0,T]$ with probability 1. One can easily check that $v$ is adapted. Since $\Phi(\cdot):=\langle \bv(\cdot), w\rangle $ is continuous for any $w\in H$
it follows from the adaptedness of $v$ that $\Phi$ is progressively measurable for any $w\in H$.
Hence, necessarily $v$ ias also an $H$-valued progressively measurable process.

Now by setting $u:=v+M$ we easily see that $u$ is a progressively measurable solution to \eqref{lin-eq-2}. Arguing as in the proof of Lemma \ref{Appendix-DET} we can show that
$u$ is unique. It also follows from the definition of $u$ above that $u\in L^\infty(0,T;H)\cap L^2(0,T;V)$, and $u\in C_w(0,T;H) $ with probability 1.}
\end{proof}
\section*{Acknowledgments}
E.~ Hausenblas and P.~A.~Razafimandimby are funded by the FWF-Austrian Science Fund through the projects P21622 and M1487, respectively.
The research on this paper was initiated during the visit of H.~Bessaih to the Montanuniversit\"at Leoben in June 2013.
She would like to thank the Chair of Applied Mathematics at the Montanuniversit\"at Leoben for hospitality. Part of this paper was written during Razafimandimby's visit at
 the University of Wyoming in November 2013. He is very grateful for the warm and kind hospitality of the Department of Mathematics at the University of Wyoming.

% Such a sequence can bi


\begin{thebibliography}{99}
% \bibitem{AS+ZB+JLW} S.~Albeverio,Z.~Brze\'zniak and J.-L.~Wu.\newblock{
% Existence of global solutions and invariant measures for stochastic differential equations driven by Poisson type noise with non-Lipschitz coefficients.}
% \newblock{\em J. Math. Anal. Appl.} \textbf{371}(1):309-322, 2010.

\bibitem{BBBF06}
D. Barbato, M. Barsanti, H. Bessaih, \& F. Flandoli,  Some rigorous results
on a stochastic Goy model, {\em Journal of Statistical Physics}, {\bf 125} (2006)
  677--716.

% \bibitem{gruen}
% J.~Becker, G.~Gr\"un, R.~Seemann, H.~Mantz, and K.~Jacobs.
% \newblock {Complex dewetting scenarios captured by thin-film models}.
% \newblock {\em Nature}, January 2003:59--63, 2006.

\bibitem{bensoussan} A.~Bensoussan. \newblock Stochastic {N}avier-{S}tokes {E%
}quations. \newblock {\em Acta Applicandae Mathematicae}, 38:267--304, 1995.

\bibitem{Bensoussan} A.~Bensoussan and J.~Freshe. Local solutions for stochastic Navier Stokes equations.
 \newblock{\em M2AN Math. Model. Numer. Anal.} \textbf{34}(2):241-273, 2000.

\bibitem{Bensoussan-Temam} A.~Bensoussan and R.~Temam. \newblock Equations {S%
}tochastiques du {T}ype {N}avier-{S}tokes.
\newblock {\em Journal of
Functional Analysis}, \textbf{13}:195--222, 1973.

\bibitem{bessaih}
H.~Bessaih and A.~Millet.
 \newblock Large deviation
principle and inviscid shell models.
 \newblock {\em Electron. J. Probab.} 14:2551-2579, 2009.

\bibitem{Flandoli3}H.~Bessaih, F.~Flandoli and E.S.~Titi. Stochastic attractors for shell phenomenological models of turbulence.
 \newblock{\em J. Stat. Phys.} 140:688--717, 2010.
%
%  \bibitem{Brzezniak3} Z.~Brzezniak and L.~Debbi. On stochastic Burgers
% equation driven by a fractional Laplacian and space-time white noise. %
% \newblock{\em Stochastic differential equations: Theory and applications,}
% Interdiscip. Math. Sci., 2, World Sci. Publ., pages 135--167, 2007.

\bibitem{Brz+Haus_2009} Z. Brze\'{z}niak, E. Hausenblas. \newblock{Maximal regularity for stochastic convolutions driven by L\'evy processes},
 \newblock{\em Probab. Theory Related Fields.}  \textbf{145}(3-4):615--637, 2009.

 \bibitem{SLC-JAP} Z.~Brze\'zniak, E.~Hausenblas, and P.~Razafimandimby. Stochastic Nonparabolic dissipative systems modeling the flow of Liquid Crystals: Strong solution.
\newblock{\em Preprint arXiv:1310.8641. To appear in
RIMS K\^oky\^uroku ``Proceeding of RIMS Symposium on Mathematical Analysis of Incompressible Flow, February 2013''}, 2013.

 \bibitem{ZB+EH+JZ}
 Z.~Brze\'zniak,E.~Hausenblas and J.~Zhu. \newblock{2D stochastic Navier-Stokes equations driven by jump noise.}
 \newblock{\em Nonlinear Anal.} \textbf{79}:122-139, 2013.

 \bibitem{Brz+Millet_2012}  Z.~Brze{\'z}niak and A. Millet, \newblock{On the stochastic Strichartz estimates and the stochastic nonlinear Schr\"odinger equation on a compact riemannian manifold},
 \textit{arXiv:1209.3578}, 2012, {\em to appear in Pot. Analysis}


\bibitem{Motyl-2}Z.~Brze\'zniak and E.~Motyl. Existence of a martingale solution of the stochastic Navier-Stokes equations in unbounded 2D and 3D domains.
\newblock{\em J. Differential Equations} \textbf{254}(4): 1627-1685, 2013.

 \bibitem{ZB et al 2005}Z.~Brze\'zniak, B.~Maslowski and J.~Seidler,  \newblock{Stochastic nonlinear beam equations},  \newblock{\em Probab. Theory Related
Fields}, \textbf{132}(1):119-149 ,2005.
%
% \bibitem{Caraballo} T.~Caraballo, J.~Real and T.~Taniguchi. {{On the
% existence and uniqueness}} of solutions to stochastic three-dimensional
% Lagrangian averaged Navier-Stokes equations. \newblock{\em Proc. R. Soc.
% Lond. Ser. A Math. Phys. Eng. Sci.} 462(2066):459--479, 2006.


\bibitem{titi2}
V. Chepyzhov, E. Titi \& M. Vishik, { On the convergence of solutions
of the Leray-$\alpha$ model to the trajectory attractor of the 3D
Navier-Stokes system}, {\em Discrete Contin. Dyn. Syst.}   \textbf{17} (2007),
481--500.

\bibitem{titi1}
A. Cheskidov, D. Holm, E. Olson \& E. Titi, { On a Leray-$\alpha$
model of turbulence}, {\em Proc. R. Soc. Lond. Ser.A}  \textbf{461} (2005),
629--649.


\bibitem{CLT06}
P. Constantin, B. Levant, \& E. S. Titi,
Analytic study of the shell model of turbulence,
{\em Physica D} {\bf 219} (2006),
120--141.
%
% \bibitem{Chow+Kashminskii}P.-L.~Chow and R.~Z.~Khasminskii.\newblock{
% Stationary solutions of nonlinear stochastic evolution equations.}
% \newblock{\em Stochastic Anal. Appl.} 15(5):671-699, 1997.

\bibitem{Chueshov}I.~Chueshov and A.~Millet.
\newblock{\em Stochastic 2D hydrodynamical type systems: well posedness and large deviations}. \newblock{\em Appl. Math. Optim.} \textbf{61}(3):
379-420, 2010.

%
% \bibitem{DAPRATO} G.~Da Prato and A.~Debussche {2D stochastic Navier-Stokes
% equations with a time-periodic forcing term.} \emph{J. Dynam. Differential
% Equations} 20(2):301--335, 2008.
%
% \bibitem{daprato} G.~Da~Prato and J.~Zabczyk.
% \newblock {\em Stochastic
% {E}quations in {I}nfinite {D}imensions}. \newblock Cambridge University
% Press, 1992.
% %
% \bibitem{Brz+Millet_2012}  Z.~Brze{\'z}niak and A. Millet, \textit{On the stochastic Strichartz estimates and the stochastic nonlinear Schr\"odinger equation on a compact riemannian manifold},
%  arXiv:1209.3578, 2012, to appear in Pot. Analysis
%
% \bibitem{Chandrasekhar}S.~Chandrasekhar, {\sc Liquid Crystals}, Cambridge University
% Press, 1992.
%
% \bibitem{Rojas-Medar2} B. Climent-Ezquerra, F. Guill\'en-Gonz\'alez and M. A. Rojas-Medar,
% \textit{Reproductivity for a nematic liquid crystal model}, \newblock  Z. Angew. Math. Phys. \textbf{57}(06):984-998 (2006).

\bibitem{deBouard+Deb_1999} A. de Bouard and A. Debussche,
 \textit{ A stochastic nonlinear Schr\"odinger equation with multiplicative noise},
 Comm. Math. Phys.  \textbf{205}(1):161--181 (1999).

\bibitem{deBouard+Deb_2003} A. de Bouard and A. Debussche,
  \textit{The stochastic nonlinear Schr\"odinger equation in $H\sp 1$},
 Stochastic Anal. Appl.  \textbf{21}(1):97--126  (2003).


\bibitem{DEUGOUE2} G.~Deugoue and M.~Sango. On the Strong Solution for the
3D Stochastic Leray-Alpha Model, \textit{Boundary Value Problems}, vol.
2010, Article ID 723018, 31 pages, 2010. doi:10.1155/2010/723018.

\bibitem{DongNS} Z. Dong and Z.~Jianliang. {Martingale solutions and Markov selection of stochastic 3D
Navier-Stokes equations with jump}, {\em  J. Differential
Equations}, \textbf{250}:2737-2778, 2011.



\bibitem{DL-mhd72}
G. Duvaut \& J.L. Lions,  In\'equations en thermo\'elasticit\'e et magn\'eto hydrodynamique,
{\em Arch. Rational Mech. Anal.} {\bf 46} (1972), 241--279.

\bibitem{MR838085}
S.~Ethier and T.~Kurtz.
\newblock {\em Markov processes, Characterization and convergence}.
\newblock Wiley Series in Probability and Mathematical Statistics: Probability
  and Mathematical Statistics. John Wiley \& Sons Inc., New York, 1986.

% \bibitem{Fabian}M.~Fabian, P.~Habala, P.~H\'ajek, V.~Montesinos and V.~Zizler. {\em Banach space theory. The basis for linear and nonlinear analysis.}
% CMS Books in Mathematics/Ouvrages de Math\'ematiques de la SMC.
% Springer, New York, 2011.

\bibitem{Flandoli2} F.~Flandoli, M.~Gubinelli, M.~Hairer, and M.~Romito.
\newblock{Rigourous remarks about scaling laws in turbulent
fluid.} \newblock{\em Commun. Math. Phys.} 278: 1-29, 2008.

\bibitem{Flandoli+Gatarek} F.~Flandoli and D.~Gatarek. Martingale and stationary solutions for stochastic Navier-Stokes equations.
\newblock{ \em Probab. Theory Related Fields.} 102(03);367-391, 1995.


\bibitem{Foias} C. Foias, O. Manley \& R. Temam,
Attractors for the B\'{e}nard problem: existence and physical
bounds on their fractal dimension. \emph{Nonlinear Analysis}
\textbf{11} (1987), 939--967.

\bibitem{GaPa} G.P. Galdi \& M. Padula,  A new approach to energy theory
in the stability of fluid motion,
{\em Arch. Rational Mech. Anal. } {\bf 110} (1990), 187--286.


 \bibitem{Gyongy+Krylov}I.~Gy\"ongy and N.~V.~Krylov, \newblock{On stochastics equations with respect to semimartingales. II. It\^o formula in Banach spaces.}
\newblock{\em Stochastics} \textbf{6}(3-4):153-173,  1981/82.

%\bibitem{Hasselblatt} B.~Hasselblatt and A.~Katok. {\em A first course in dynamics. With a panorama of recent developments.} Cambridge University Press, New York, 2003.
%
%  \bibitem{Bipolar}E.~Hausenblas and P.~A.~Razafimandimby. On stochastic
% evolution equations for nonlinear bipolar fluids: well-posedness
% and some properties of the solution. { \em arXiv:1206.1172}, 2012.

 \bibitem{paulandme2}E.~Hausenblas, P.~A.~Razafimandimby and M.~Sango. Martingale
solution to differential type fluids of grade two driven by random
force of L\'evy type.  \newblock{\em Potential Anal.} 38(4):1291-1331, 2013.

%
%  \bibitem{Holst+et al}
%  M.~Holst, E.~Lunasin, and G.~Tsogtgerel, \newblock{ Analysis of a general family of regularized Navier-Stokes and MHD models.}
%  \newblock{\em J. Nonlinear Sci.} \textbf{20}(5):523--567, 2010.


\bibitem{MR2205118}
P.~Imkeller and I.~Pavlyukevich.
\newblock First exit times of {SDE}s driven by stable {L}\'evy processes.
\newblock {\em Stochastic Process. Appl.}, 116:611--642, 2006.


\bibitem{KP05}
 N. H. Katz \& N. Pavlovi\'c. Finite time blow-up for a dyadic model
 of the Euler equations, {\em Trans. Amer. Math. Soc.} {\bf 357} (2005), 695--708.


\bibitem{KIM}J.~U.~Kim.
Strong solutions of the stochastic Navier-Stokes equations in $\RR^3$. \newblock{\em Indiana Univ. Math. J.}
\textbf{59}(4):1417-1450, 2010.




 \bibitem{KAP} A.~Kupiainen. Statistical theories of turbulence. In \newblock{\em Advances in Mathematical Sciences and Applications.}
Gakkotosho, Tokyo, 2003.


 \bibitem{LadSol-mhd60}
O. Ladyzhenskaya \& V. Solonnikov, {Solution of some nonstationary
magnetohydrodynamical problems for incompressible fluid},
{\em Trudy  Steklov Math. Inst.} {\bf 59} (1960), 115--173; in Russian.

\bibitem{LPPPV98}
V. S. Lvov, E. Podivilov, A. Pomyalov, I. Procaccia \& D. Vandembroucq,
Improved shell model of turbulence,
{\em Physical Review E}, {\bf 58} (1998), 1811--1822.

% \bibitem{manna} U.~Manna and M. T.~ Mohan. Shell model of turbulence perturbed by \levy noise
% \newblock{\em Nonlinear Differ. Equ. Appl.} 18:615--648, 2011.

\bibitem{MIKU}R.~Mikulevicius. On strong $H^\frac12$-solutions of stochastic Navier-Stokes equation in a bounded domain.
\newblock{\em SIAM J. Math. Anal.} \textbf{41}(3):1206-1230, 2009.


\bibitem{Motyl}E.~Motyl. Stochastic Navier-Stokes equations driven by L\'evy noise in unbounded 3D domains.
\newblock{\em Potential Anal.} \textbf{38}(3):863-912, 2013.


\bibitem{G-H+Z} N.~Glatt-Holtz and M.~Ziane.
Strong pathwise solutions of the stochastic Navier-Stokes system. \newblock{\em Adv. Differential Equations} \textbf{14}(5-6):567-600, 2009.

\bibitem{OY89}
K. Ohkitani \&  M.  Yamada, Temporal intermittency in the energy cascade process
and local Lyapunov analysis in fully developed model of turbulence,
{\em Prog. Theor. Phys.}
{\bf 89} (1989), 329--341.

%
%  \bibitem{Pardoux}E.~Pardoux,  \newblock{ Stochastic partial differential equations and
% filtering of diffusion processes},  {\em Stochastics} \textbf{3}(2):127-167, 1979.

\bibitem{Pardoux-75}E.~Pardoux, \newblock{\em Equations aux d\'eriv\'ees partielles stochastiques non lin\'eaires monotones; Etude de solutions fortes de type It\^o}.
Th\`ese (PhD Thesis), Universit\'e Paris Sud, 1975.

\bibitem{Pes+Zab}S.~ Peszat and J.~Zabzcyk. {\em Stochastic partial differential equations with L\'evy noise.
An evolution equation approach.} Volume \textbf{113} of  {\em Encyclopedia of Mathematics and its Applications}. Cambridge University Press, Cambridge, 2007.



\bibitem{Protter}P.~E.~Protter. {\em Stochastic integration and differential equations.}
Second edition. Volume 21 of {\em Applications of Mathematics} (New York). {\em Stochastic Modelling and Applied Probability}. Springer-Verlag, Berlin, 2004.


\bibitem{ROZOVSKII1} R.~Mikulevicius and B.L.~ Rozovskii. %
\newblock{Stochastic {N}avier-{S}tokes {E}quations and {T}urbulent {F}lows.} %
\newblock{\em SIAM J. Math. Anal.}, 35(5):1250-1310, 2004.

\bibitem{Rud+Zig} B.~R\"udiger and G.~Ziglio. \newblock{It\^o formula for stochastic integrals w.r.t. compensated Poisson random measures on separable Banach spaces.}
\newblock{\em Stochastics.} \textbf{78}(6):377--410.
%
% \bibitem{sango2} M.~Sango. Magnetohydrodynamic turbulent flows: Existence
% results.\textit{\ Physica D: Nonlinear Phenomena} 239(12): 912-923, 2010.
%
% \bibitem{sango} M.~Sango. Density dependent stochastic Navier-Stokes
% equations with non Lipschitz random forcing.  \textit{Reviews in
% Mathematical Physics} 22(6):669--697, 2010.

\bibitem{SeTe} M. Sermange \& R. Temam, Some mathematical questions related
to MHD equations, {\em Communications in Pure ad Applied Mathematics} {\bf 36}
(1983), 635--664.

% \bibitem{Taniguchi}T.~Taniguchi. \newblock{The existence and asymptotic behaviour of energy solutions to stochastic
% 2D functional Navier-Stokes equations driven by Levy processes.}
% \newblock{\em Journal of Mathematical Analysis and Applications}. 385(2):634-654, 2012.

%\bibitem{Strauss} W.~A.~Strauss, \newblock{On continuity of functions with values in various Banach spaces.} \newblock{Pacific J. Math.} \textbf{19}: 543--551, 1966.
\end{thebibliography}
\end{document}